\documentclass{amsart}

\usepackage{amsmath,amssymb,amsthm}
\usepackage{mathrsfs}
\usepackage[margin=1.25in]{geometry}
\usepackage{graphicx}
\usepackage{subfigure}
\usepackage{float}
\usepackage{xcolor}
\usepackage{hyperref}

\newcommand{\bx}{\boldsymbol{x}}
\newcommand{\bR}{\mathbb{R}}
\newcommand{\bu}{\mathbf{u}}
\newcommand{\bv}{\mathbf{v}}
\newcommand{\bw}{\mathbf{w}}
\newcommand{\bxit}{\mathbf{x}}
\newcommand{\bz}{\mathbf{z}}
\newcommand{\bbM}{\mathbb{M}}
\newcommand{\bbS}{\mathbb{S}}
\newcommand{\bbV}{\mathbb{V}}
\newcommand{\bE}{\mathbf{E}}
\newcommand{\bg}{\mathbf{g}}

\newcommand{\bn}{\mathbf{n}}
\newcommand{\bdelta}{\boldsymbol{\delta}}
\newcommand{\bxi}{\boldsymbol{\xi}}

\newcommand{\calH}{\mathcal{H}}
\newcommand{\calG}{\mathcal{G}}

\newcommand{\bp}{\mathbf{p}}
\newcommand{\bd}{\mathbf{d}}
\newcommand{\bbG}{\mathbb{G}}
\newcommand{\bbI}{\mathbb{I}}
\newcommand{\calM}{\mathcal{M}}
\newcommand{\calP}{\mathcal{P}}
\newcommand{\calL}{\mathcal{L}}
\newcommand{\calR}{\mathcal{R}}
\newcommand{\calT}{\mathcal{T}}
\newcommand{\norm}[1]{\left\| #1\right\|}
\DeclareMathOperator{\diag}{diag}

\numberwithin{equation}{section}
\newtheorem{theorem}{Theorem}[section]
\newtheorem{lemma}[theorem]{Lemma}
\newtheorem{proposition}[theorem]{Proposition}

\newtheorem{assumption}[theorem]{Assumption}

\theoremstyle{remark}
\newtheorem{remark}[theorem]{Remark}
\newtheorem{example}[theorem]{Example}

\title[Global convergence of a splitting method for the GP ground state]{Global convergence of an efficient splitting method for the defocusing Gross--Pitaevskii ground state problem}

\author{Jiaxing Li}
\author{Shixin Zheng}
\author{Xiangxiong Zhang}

\thanks{J. Li and X. Zhang: Department of Mathematics, Purdue University, 150 N. University Street, West Lafayette, IN 47907 (\texttt{li4944@purdue.edu}, \texttt{zhan1966@purdue.edu}).}
\thanks{S. Zheng: Department of Mathematics, University of Maryland, College Park, MD 20742 (\texttt{siuzheng@umd.edu}).}
\thanks{XZ was supported in part by NSF DMS-2208518.}

\date{\today}

\newif\ifnnnfigs
\nnnfigstrue

\begin{document}

\begin{abstract}
For computing the ground state of the defocusing Gross--Pitaevskii
energy, we propose and analyze two efficient schemes based on the
Davis--Yin three-operator splitting, which treats the potential and interaction terms
explicitly, and the kinetic energy by a resolvent. One iteration
costs one inversion of $\bbI-\gamma\Delta$ for the first scheme,
and of $\bbI-\gamma\Delta+\gamma V_1$ for the second, with $V_1$
denoting the separable part of the potential, and on structured
meshes both operators can be inverted by simple fast GPU solvers. For
monotone discrete Laplacians, including the second-order finite
difference scheme and the lumped linear finite element method on
simplicial meshes with suitable angle conditions, we prove global
convergence to the unique positive discrete ground state, for every
positive normalized initial vector, for any constant step size
below an explicit threshold. In contrast, the methods previously
proven to converge globally to the ground state all invert a more
difficult elliptic operator that depends on the iteration variable. In three-dimensional tests with up to $999^3$ unknowns
on one GPU, a simple variable step size rule makes the proposed splitting
schemes efficient in practice, comparable in wall-clock time to
Riemannian conjugate gradient methods that also invert only a
shifted Laplacian operator, and much more robust with respect to
the choice of the initial guess.
\end{abstract}

\maketitle

{\footnotesize
\smallskip

\par}

\section{Introduction}
\label{sec:intro}

\subsection{The Gross--Pitaevskii ground state problem}
\label{sec:gp}

A standard  model for the equilibrium states of a
Bose--Einstein condensate
\cite{pitaevskii2003bose}
is the minimization of the Gross--Pitaevskii (GP) energy. With a standard
rescaling (see, e.g., \cite{lieb2001bosons,bao2013mathematical}) the
GP ground state problem on a bounded domain
$\Omega\subset\bR^d$, for instance $\Omega=[-L,L]^d$, can be stated
as minimizing the energy functional
\begin{equation}
E(\phi)=\frac12\int_\Omega\left(|\nabla\phi(\bx)|^2
+V(\bx)|\phi(\bx)|^2\right)\mathrm{d}\bx
+\frac{\beta}{4}\int_\Omega|\phi(\bx)|^4\,\mathrm{d}\bx,
\label{GPenergy}
\end{equation}
over the constraint set
\begin{equation}
\left\{\phi\in H_0^1(\Omega):\ \int_\Omega|\phi(\bx)|^2\,
\mathrm{d}\bx=1\right\},
\label{GPconstraint}
\end{equation}
where $d=1,2,3$ is the dimension, $V(\bx)$ is  a given potential, and
$\beta$ is the coupling constant. We only consider the case $\beta>0$, which is the
\emph{defocusing} (repulsive) regime. 

The existence, uniqueness, and regularity of the GP ground state for $\beta>0$ are
well understood, see, e.g., \cite{lieb2001bosons}. For $\beta>0$ the
minimizer of \eqref{GPenergy}--\eqref{GPconstraint} is unique up to
sign, and its positive representative $u^*>0$, called \emph{the}
ground state, is an eigenfunction of the nonlinear eigenvalue
problem
\begin{equation}
-\Delta u(\bx)+V(\bx)u(\bx)+\beta|u(\bx)|^2u(\bx)=\lambda u(\bx),
\qquad
\int_\Omega|u(\bx)|^2\,\mathrm{d}\bx=1,\qquad
u(\bx)|_{\partial\Omega}=0.
\label{continuum}
\end{equation}
Since a constant shift of the potential changes
the energy \eqref{GPenergy} only by a constant over the constraint set
\eqref{GPconstraint}, the assumption $V\ge0$ loses no generality.
Here \eqref{continuum} is understood in the sense of distributions,
i.e., in the variational form: seek $\lambda\in\bR$ and
$u\in H_0^1(\Omega)$ satisfying
\begin{equation}
(\nabla u,\nabla v)+(Vu,v)+\beta(|u|^2u,v)=\lambda\,(u,v),
\qquad\forall\,v\in H_0^1(\Omega),
\label{variational}
\end{equation}
where $(u,v)=\int_\Omega u(\bx)v(\bx)\,\mathrm{d}\bx$. Taking
$u=v=u^*$ in \eqref{variational} shows that the ground state
eigenvalue is determined by the ground state through
\begin{equation}
\lambda^*=2E(u^*)+\frac\beta2\int_\Omega|u^*(\bx)|^4\,\mathrm{d}\bx .
\label{lambdastar}
\end{equation}

\subsection{Numerical methods and their convergence theory}
\label{sec:relatedwork}

After a spatial discretization,  the ground state can be solved by
minimizing the discrete GP energy over the convex sphere constraint in $\bR^N$. We refer to the survey
\cite{bao2013mathematical} for models and methods, and to the recent
review  \cite{henning2025gross} for a
systematic account of the convergence theory, whose terminology we
follow. \emph{Algebraic methods} discretize first and then iterate
on the finite-dimensional problem: the self-consistent field (SCF)
iteration
\cite{defranceschi2000scf,cances2000convergence,upadhyaya2018density},
the inverse iteration or $\mathscr A$-method
\cite{jarlebring2014inverse}, and the $J$-method
\cite{altmann2021j}, which linearizes with the Jacobian of the
nonlinear eigenvalue problem and admits spectral shifts.
\emph{Variational methods} discretize an iteration derived in
function space: the gradient flow with discrete normalization (GFDN) \cite{bao2004computing}, a semi-implicit backward
Euler step for the projected $L^2$-gradient flow and the Sobolev gradient flows obtained by
representing the gradient in other metrics, namely the $H^1$- and
$a_0$-flows
\cite{kazemi2010minimizing,danaila2010new,danaila2017computation,chen2023convergence}
and the energy-adaptive $a_u$-flow
\cite{henning2020sobolev,altmann2022energy}.
Riemannian optimization with conjugate gradient acceleration
\cite{danaila2017computation,antoine2017efficient} and Newton-type
methods \cite{wu2017regularized,altmann2024riemannian} are faster
in practice, but the conjugate gradient variants come without a
convergence theory \cite[\S5.3]{henning2025gross}, apart
from the discussion in \cite[\S5.4]{antoine2017efficient}, which sketches an
argument for convergence to a critical point and gives heuristics for
the rate,
and among the Newton-type methods the available
guarantee is the convergence of the residual to zero for
the regularized Newton method \cite{wu2017regularized}, by a
trust-region argument, which yields a stationary point rather than
the ground state.

Convergence of these methods can be summarized into three levels.  

\emph{(i) Local convergence to the ground state.} Many methods have
been proven to converge locally at a linear rate governed by a
spectral gap of a linearized operator, such as SCF
\cite{upadhyaya2018density}, the $\mathscr A$-method
\cite{jarlebring2014inverse}, the $J$-method \cite{altmann2021j},
and the Sobolev gradient flows
\cite{zhang2019exponential,henning2020sobolev,henning2023dependency,chen2023convergence,chen2024fully}.

\emph{(ii) Global convergence to a critical point.} From any initial
guess, a monotone energy decay combined with a summability argument
shows that every accumulation point is a constrained critical point
of the energy, as shown for the $H^1$- and $a_0$-flows
\cite{chen2023convergence,chen2024fully}. For the rotating energy
functional and for two-component condensates, energy dissipation
and global convergence to a critical point were established in
\cite{feng2025preconditioned,zhang2025convergence,feng2025two}.

\emph{(iii) Global convergence to the ground state.} The few results
at this level all rest on positivity as a selection mechanism.
Henning and Peterseim \cite{henning2020sobolev} proved that the
$a_u$-flow started from a nonnegative initial guess converges in
$H^1$ to a strictly positive eigenfunction, necessarily the ground
state, since the Gross--Pitaevskii energy admits no positive excited
state. The same conclusion was reached for the damped $J$-method
\cite[Prop.~3]  {altmann2021j}. For the classical GFDN even the energy-diminishing property was a
long-standing open problem \cite[\S5.2.1]{henning2025gross}. It was
settled recently by Feng, Tang and Wang \cite{feng2025discrete},
together with global convergence to a critical point, and in the
defocusing case without rotation the whole sequence started from a
nonnegative initial value converges in $H^1$ to the ground state,
by the positivity preservation of the semi-implicit step
\cite[Cor.~1]{feng2025discrete}. For the continuous-in-time
$L^2$-normalized gradient flow, which is the model behind the GFDN
rather than the algorithm itself, convergence to the ground state
from a nonnegative initial value was proven recently in
\cite{chu2025gradient}. These arguments use the uniqueness and
positivity of the \emph{continuous} ground state, properties that a
spatial discretization does not automatically inherit. They are,
however, available for the monotone discretizations of
Section~\ref{sec:fem}. For the lumped $P^1$ and finite difference
schemes, the discrete ground state is unique up to sign, strictly
positive, and the only nonnegative critical point of the discrete
energy on the sphere, by the M-matrix structure combined with
either the Perron--Frobenius theorem \cite{chen2024fully} or a
discrete Picone inequality \cite{hauck2024positivity}. With these
properties, the fully discrete $a_u$-flow for the lumped $P^1$
scheme was proven to converge to the discrete ground state for
every nonnegative normalized initial guess
\cite[Cor.~4.2]{hauck2024positivity}.

\subsection{The scheme and its computational cost}
\label{sec:scheme}

The three methods that reach level (iii), the $a_u$-flow, the
damped $J$-method and the GFDN, all invert an
elliptic operator that depends on the iteration variable, and no method whose iteration inverts
only a shifted Laplacian has been proven
to converge globally to the ground state, e.g., it is quite
difficult to prove that the $H^1$- or $a_0$-scheme can preserve
positivity \cite[\S3.1]{chen2023convergence}. Motivated by
this gap, we propose and analyze a semi-implicit splitting method for the
discrete ground state problem \eqref{gs-h} and prove its global
convergence to the unique positive discrete ground state. By applying the Davis--Yin three-operator splitting
\cite{davis2017three} to natural three-part decompositions of
\eqref{gs-h}, and by derivation in Appendix~\ref{app:dys}, with
notation in Section~\ref{sec:fem}, we obtain two schemes:
\begin{equation}
\begin{cases}
\bu^k=\dfrac{\bz^k}{\norm{\bz^k}_2},\\[8pt]
\bxit^k=(\bbI-\gamma\Delta_h)^{-1}
\bigl(2\bu^k-\bz^k-\gamma(\bbV\bu^k+\beta(\bu^k)^3)\bigr),\\[6pt]
\bz^{k+1}=\bz^k+\bxit^k-\bu^k ,
\end{cases}
\qquad\text{(DYS~I)}
\label{dys}
\end{equation}
\begin{equation}
\begin{cases}
\bu^k=\dfrac{\bz^k}{\norm{\bz^k}_2},\\[8pt]
\bxit^k=(\bbI-\gamma\Delta_h+\gamma\bbV_1)^{-1}
\bigl(2\bu^k-\bz^k-\gamma(\bbV_2\bu^k+\beta(\bu^k)^3)\bigr),\\[6pt]
\bz^{k+1}=\bz^k+\bxit^k-\bu^k ,
\end{cases}
\qquad\text{(DYS~II)}
\label{dys2}
\end{equation}
where $\gamma>0$ is the step size, $V=V_1+V_2$ with $V_1\ge0$ and
$V_2\ge0$, $\bbV_1$ and $\bbV_2$ are the corresponding diagonal
matrices, and $\Delta_h$ denotes the discrete Laplacian.

We refer to the first scheme  \eqref{dys} as DYS~I in which only a shifted Laplacian is inverted.  When $V_1$ is \emph{separable},
 i.e., a sum of one-dimensional potentials, each depending on a single coordinate $x_j$, as is the trapping potential $|\bx|^2=\sum_{j=1}^{d}x_j^2$,  
 we can also consider the second scheme \eqref{dys2}, referred to as
DYS~II, in which
$(\bbI-\gamma\Delta_h+\gamma\bbV_1)^{-1}$ can still be inverted
easily and efficiently on GPUs on structured meshes, by the same
per-axis eigendecomposition that inverts $-\Delta_h$ directly, see
\cite{liu2026gpu}.   

From an implementation point of view, the methods reviewed in
\S\ref{sec:relatedwork} differ mainly in the linear system solved
at each iteration, compared in Table~\ref{tab:cost}. In general, a shifted Laplacian is much
easier to invert than an elliptic operator containing a spatially
varying potential term.

\begin{table}[htbp]
\centering
\small
\begin{tabular}{llcl}
\hline
method & operator to invert & solves/iter.\ & convergence proved\\
\hline
GFDN
& $-\Delta_h+\bbV+\beta\diag(\bu^2)+\tau^{-1}\bbI$
& 1 & ground state \cite{feng2025discrete}\\
damped $J$
& $-\Delta_h+\bbV+3\beta\diag(\bu^2)-\sigma\bbI$ (+ rank one)
& 2 & ground state \cite{altmann2021j}\\
$a_u$-flow
& $-\Delta_h+\bbV+\beta\diag(\bu^2)$
& 1 & ground state \cite{hauck2024positivity,henning2020sobolev}\\
$a_0$-flow
& $-\Delta_h+\bbV$
& 2 & critical point \cite{chen2023convergence}\\
$H^1$-flow
& $-\Delta_h+\alpha\bbI$
& 2 & critical point \cite{chen2024fully}\\
DYS~I
& $\bbI-\gamma\Delta_h$
& \textbf{1} & \textbf{ground state} (Theorem~\ref{thm:main})\\
\hline
\end{tabular}
\caption{Elliptic solvers needed by one iteration, and the
strongest convergence result proved for each method:
\emph{ground state} and \emph{critical point} mean global
convergence, levels (iii) and (ii) of \S\ref{sec:relatedwork}.}
\label{tab:cost}
\end{table}

\subsection{Contributions of the present work}
\label{sec:contrib}

The contributions of the present work are twofold. First, we
propose the two efficient splitting schemes DYS~I and DYS~II of
\eqref{dys} and \eqref{dys2}, whose single solve per iteration
involves only $\bbI-\gamma\Delta_h$ or, with $\bbV_1$ the
separable part of the potential,
$\bbI-\gamma\Delta_h+\gamma\bbV_1$. Second, we prove global
convergence to the unique positive discrete ground state for
\eqref{dys}, to our knowledge the first such result for a method
that inverts only a shifted Laplacian. Our main result
(Theorem~\ref{thm:main}) states that, if using a monotone discrete
Laplacian, then for every positive normalized initial guess
$\bu^0>0$ and every step size $0<\gamma\le\gamma_{\max}$, with
$\gamma_{\max}>0$ explicitly computable from the data, the
iteration \eqref{dys} converges to the unique positive discrete
ground state.  Once the iterates are close to the ground
state, the convergence is linear (Theorem~\ref{thm:rate}). 
For simplicity we prove the theorem only for DYS~I,
but all arguments apply to DYS~II, since
$\bbI-\gamma\Delta_h+\gamma\bbV_1$ differs from
$\bbI-\gamma\Delta_h$ by the nonnegative diagonal matrix
$\gamma\bbV_1$, which preserves the M-matrix structure of the negative discrete Laplacian and the
entrywise positive inverse used in the analysis, see Remark~\ref{d2:rem:dys2}. 
Alongside the theorem, Section~\ref{sec:numerics}
compares \eqref{dys} and its variant \eqref{dys2} with five
efficient methods on three-dimensional problems with up to $999^3$
unknowns on one GPU: with a simple step size rule, the Davis--Yin schemes match the conjugate gradient methods
in wall-clock time and are much more robust with respect to the
choice of the initial guess.

The general convergence theory for the Davis--Yin
splitting requires all three parts of the objective to be convex
\cite{davis2017three}, so the
convergence of \eqref{dys} has to be established from scratch due to  the unit sphere constraint.
The proof of Theorem~\ref{thm:main} overcomes two difficulties.
First, the energy of DYS I is not guaranteed to decrease
monotonically, thus instead we analyze  the radial
direction as well as the tangential one, through a Lyapunov
function that couples the energy of the normalized iterate with
a radial residual
$R^k=\norm{\bz^k}_2-1+\gamma\lambda(\bu^k)$, where $\lambda(\bu^k)$
is the discrete eigenvalue functional of \eqref{d2:grad}.
Second, one has to prove that the iterates stay
positive, and the proof is significantly different from and much
more difficult than those for the $a_u$-flow and the GFDN, since one
step of \eqref{dys} reads
$\bz^{k+1}=(\norm{\bz^k}_2-1)\bu^k+\bxit^k$, and the first term is
negative whenever $\norm{\bz^k}_2<1$.

\subsection{Organization of the paper}
\label{sec:organization}

Section~\ref{sec:fem} presents the discrete ground state problem and the
structural properties used in the paper.
Section~\ref{sec:setting} introduces the quantities in which the
analysis is carried out, states the main theorem, and describes the
three steps of our proof strategy.  Section~\ref{sec:onestep} derives two exact identities
for one step, proves positivity, the radial interval and a
conditional descent estimate for the Lyapunov function, and removes
the condition by an energy barrier argument.
In Section~\ref{sec:convergence}, we prove the global convergence to the
discrete ground state, and also prove a local linear rate. 
Numerical tests are given in Section~\ref{sec:numerics},
and concluding remarks are in Section~\ref{sec:conclusion}.
 
\section{The discrete ground state problem}
\label{sec:fem}

We discretize \eqref{continuum} in space by the classical continuous
finite element method with quadrature, following \cite{chen2024fully},
using either $Q^1$ elements with Gauss--Lobatto quadrature on a
uniform rectangular mesh, which coincides with the classical
second-order finite difference scheme, or $P^1$ elements with mass
lumping on a simplicial mesh satisfying a suitable angle
condition.  See \cite{chen2024fully} for their construction, the precise mesh
condition, and the monotonicity properties that follow, which are
also briefly reviewed in Appendix~\ref{app:fem}.  

Let $\bx_i$ ($i=1,\dots,N$) be the interior nodes of the mesh, let
$w_i>0$ be the quadrature weight (in either scheme above) at $\bx_i$, let $\phi_i$ be the nodal
basis functions, and let $\langle\cdot,\cdot\rangle$ denote the
quadrature approximation of the $L^2$ inner product. A 
function is represented by its nodal values
$\bu=\begin{bmatrix}u_1&\cdots&u_N\end{bmatrix}^\top\in\bR^N$. Set
\[
\bbM=\diag\{w_1,\cdots,w_N\},\qquad
\bbV=\diag\{V_1,\cdots,V_N\}\ \ \text{with}\ \ V_i=V(\bx_i),
\qquad
\bbS_{ij}=\langle\nabla\phi_i,\nabla\phi_j\rangle,
\]
so that $\bbM$ is the lumped mass matrix and $\bbS$ the stiffness
matrix. For a symmetric matrix $A$ we write $A\succeq0$ for positive
semidefiniteness and $A\preceq B$ for $B-A\succeq0$. The discrete
counterpart of \eqref{variational} is the
nonlinear eigenvalue problem
\begin{equation}
\bbS\bu+\bbM\bbV\bu+\beta\bbM\bu^3=\lambda_h\bbM\bu ,
\label{fd3}
\end{equation}
where $\bu^3:=\begin{bmatrix}u_1^3&\cdots&u_N^3\end{bmatrix}^\top$. Equivalently, with the discrete Laplacian
$\Delta_h:=-\bbM^{-1}\bbS$,
\begin{equation}
-\Delta_h\bu+\bbV\bu+\beta\bu^3=\lambda_h\bu .
\label{fd2}
\end{equation}
For $\bv\in\bR^N$ set $\norm{\bv}_p^p=\sum_iw_i|v_i|^p$ for
$1\le p<\infty$ and $\norm{\bv}_\infty=\max_i|v_i|$. For the case $p=2$, we denote the discrete
$L^2$ inner product and norm by
\begin{equation}
\langle\bu,\bv\rangle_h:=\bu^\top\bbM\bv ,
\qquad
\norm{\bu}_2:=\langle\bu,\bu\rangle_h^{1/2}.
\label{discreteL2norm}
\end{equation}
Thus the discrete energy is
\begin{equation}
\bE_h(\bu)
=\frac12\langle-\Delta_h\bu,\bu\rangle_h
+\frac12\langle\bbV\bu,\bu\rangle_h
+\frac\beta4\norm{\bu}_4^4 ,
\label{fd-energy}
\end{equation}
and the \emph{discrete ground state problem} is
\begin{equation}
\min_{\bu\in\bR^N}\ \bE_h(\bu)
\qquad\text{subject to}\qquad
\norm{\bu}_2=1 .
\label{gs-h}
\end{equation}
The constraint set is the manifold
$\calM:=\{\bu\in\bR^N:\norm{\bu}_2=1\}$, the unit sphere of the
discrete $L^2$ norm.
Any minimizer of \eqref{gs-h} solves \eqref{fd3}, and taking the
$\langle\cdot,\cdot\rangle_h$-product of \eqref{fd2} with $\bu$ gives,
as in \eqref{lambdastar},
$\lambda_h=2\bE_h(\bu)+\frac\beta2\norm{\bu}_4^4$. At a
fixed vector $\bu$, the linearized operator of \eqref{fd2} is
\begin{equation}
A_{\bu}=-\Delta_h+\bbV+\beta\diag(\bu^2) .
\label{eq:A_u}
\end{equation}

We assume that the spatial discretization is monotone:

\begin{assumption}[Monotone conforming discretization]\label{asp:struct}
Let $\Omega\subset\bR^d$, $d\le3$, be a bounded polytope on which
elliptic regularity holds, discretized by a shape-regular mesh, and
let $V_i\ge0$ for every $i$ and $\beta>0$. The spatial discretization
is a monotone scheme constructed from a conforming finite element
method for the Laplacian, that is, one of the two schemes of
Appendix~\ref{app:fem}, in both cases with the quadrature described
there, so that the mass matrix is diagonal:
\begin{itemize}
\item[(a)] the $Q^1$ finite element method on a uniform rectangular
mesh, with the two-point Gauss--Lobatto rule in each variable on each
cell, which coincides with the classical second-order finite
difference scheme, or
\item[(b)] the $P^1$ finite element method with the vertex rule on a
simplicial mesh satisfying the angle condition \eqref{simpicialmesh},
which in two dimensions a Delaunay triangulation satisfies.
\end{itemize} 
\end{assumption}

The properties of these schemes that the analysis uses are collected
in the following lemma, and no others are used.

\begin{lemma}[What the discretization provides]\label{lem:disc}
Let Assumption~\ref{asp:struct} hold. Then the following hold.
\begin{itemize}
\item[(i)] The mass matrix is diagonal,
$\bbM=\diag\{w_1,\dots,w_N\}$, with $w_i>0$ for every $i$.
\item[(ii)]  The stiffness matrix satisfies
$\bbS=\bbS^\top\succeq0$ and $\bbS_{ij}\le0$ for every $i\neq j$,
and for every $\gamma>0$ the inverse $(\bbM+\gamma\bbS)^{-1}$ has
strictly positive entries.
\item[(iii)] There is a constant $C_S$, independent of the
mesh size, such that, for any $\bv\in\bR^N$,
\begin{equation}
\norm{\bv}_6\ \le\ C_S\,\norm{\bv}_{1,h},
\qquad
\norm{\bv}_{1,h}^2:=\langle\bv,-\Delta_h\bv\rangle_h
+\norm{\bv}_2^2 ,
\label{sobolev}
\end{equation}
where $\norm{\cdot}_{1,h}$ is the discrete $H^1$ norm.
\end{itemize}
\end{lemma}

\begin{proof}
Item (i) is the quadrature, and the two sign conditions in item (ii)
are the monotonicity of the discrete Laplacian, established for both
schemes in Appendix~\ref{app:fem}. For scheme (b) the second is
exactly what \eqref{simpicialmesh} provides
\cite[Lemma~2.1]{xu1999monotone}. Given them, $\bbM+\gamma\bbS$ is
symmetric and positive definite, because
$\bv^\top(\bbM+\gamma\bbS)\bv
=\bv^\top\bbM\bv+\gamma\,\bv^\top\bbS\bv\ge\bv^\top\bbM\bv>0$ for
$\bv\neq0$. Its off-diagonal entries are $\gamma\bbS_{ij}\le0$ for
$i\neq j$, since $\bbM$ is diagonal, and its adjacency graph is connected. The inverse of an irreducible
Stieltjes matrix is entrywise positive
\cite[Chapter~6]{bermanplemmons}, which is the last assertion of item
(ii). Item (iii) is \cite[Lemma~5.5(iv)]{chen2024fully} together with
the norm equivalences of its appendix, which identify
$\norm{\cdot}_{1,h}$ with the $H^1$ norm and the quadrature norms
with the Lebesgue norms, all with constants independent of the mesh
size. The elliptic regularity of $\Omega$ required in
Assumption~\ref{asp:struct} enters only here, as the hypothesis under
which the equivalence between $\norm{\cdot}_{1,h}$ and the $H^1$ norm
is proven in the appendix of \cite{chen2024fully}. Conformity of the
two schemes and $d\le3$ are what make the continuous embedding
$H^1\hookrightarrow L^6$ available.
\end{proof}

\begin{remark}[The discrete ground state]\label{d2:rem:gs}
Under Assumption~\ref{asp:struct} the matrix
$\bbM A_{\bu}=\bbS+\bbM\bbV+\beta\bbM\diag(\bu^2)$ is an irreducible
symmetric M-matrix for every $\bu$. As a consequence, \eqref{gs-h}
has a minimizer that is unique up to sign, whose positive
representative $\bu_{\mathrm{GS}}$ is called the \emph{discrete ground
state}. It is the unique nonnegative minimizer of
$\bE_h$ on $\calM$, and every nonnegative critical point of
$\bE_h|_{\calM}$ coincides with it. See Appendix~\ref{app:fem} for the argument and
for references.
\end{remark}

\section{The scheme, the Lyapunov function, and the main theorem}
\label{sec:setting}

The symbols $\succeq$ and $\preceq$ are used only for
symmetric matrices such as
$\bbS$ and $\bbM+\gamma\bbS$. Most operators of this paper
 need not be  symmetric matrices, among them the discrete Laplacian
$-\Delta_h=\bbM^{-1}\bbS$, the linearized operator $A_{\bu}$, the
tangent projection $\calP_{\bu}$ and the resolvent $\bbG_X^{-1}$
introduced below. Each of them is instead \emph{$h$-self-adjoint},
meaning that $\langle\bu,A\bv\rangle_h=\langle A\bu,\bv\rangle_h$
for all $\bu,\bv$, where $\langle\cdot,\cdot\rangle_h$ is the
discrete inner product \eqref{discreteL2norm}. Since
$\langle\bu,A\bv\rangle_h=\bu^\top\bbM A\bv$, this holds exactly
when the matrix $\bbM A$ is symmetric, which implies that
such an operator has  real eigenvalues and an $\langle\cdot,\cdot\rangle_h$-orthonormal
eigenbasis. Inequalities between $h$-self-adjoint operators are
always written out as inequalities between the associated quadratic
forms, never with $\preceq$.

\subsection{Norms, energy, and gradients}\label{ssec:norms}

The tangent projection at $\bu\in\calM$ is
$\calP_{\bu}:=\bbI-\bu\bu^\top\bbM$. Writing
$w_{\min}:=\min_iw_i>0$, one has the
pointwise bound
\begin{equation}
\norm{\bv}_\infty\le w_{\min}^{-1/2}\norm{\bv}_2
\qquad(\bv\in\bR^N),
\label{d2:embed}
\end{equation}
 since $w_i|v_i|^2\le\norm{\bv}_2^2$ for every $i$. The
bound holds only since the mesh is fixed and $\bR^N$ is finite
dimensional, and there is no such bound in the continuum. Its
constant is of order $h^{-d/2}$ on quasi-uniform meshes, and it is
the weakest of the three pointwise bounds collected in
Lemma~\ref{lem:ebd0}(ii). Those bounds and the stiffness-to-mass
ratio $\Lambda_{\mathrm d}$ are the only place where the mesh enters
the step-size threshold, see Remark~\ref{d2:rem:mesh}. 

The gradient of the discrete energy \eqref{fd-energy} with respect
to $\langle\cdot,\cdot\rangle_h$ is
\begin{equation}
\nabla\bE_h(\bu)=-\Delta_h\bu+\bbV\bu+\beta\bu^3 .
\label{d2:energy}
\end{equation}
The linearized operator \eqref{eq:A_u} is
$A_{\bu}=-\Delta_h+\bbV+\beta\diag(\bu^2)$, and
$\nabla\bE_h(\bu)=A_{\bu}\bu$. The eigenvalue functional and the
tangential part of the gradient are
\begin{equation}
\lambda(\bu):=\langle\nabla\bE_h(\bu),\bu\rangle_h,
\qquad
\bg(\bu):=\calP_{\bu}\nabla\bE_h(\bu)
=\nabla\bE_h(\bu)-\lambda(\bu)\,\bu .
\label{d2:grad}
\end{equation}
A point $\bu\in\calM$ is a critical point of $\bE_h|_{\calM}$ when
$\bg(\bu)=0$, equivalently $A_{\bu}\bu=\lambda(\bu)\bu$.

By \eqref{d2:grad} the vector $\bg(\bu)$ is the projection of the
gradient $\nabla\bE_h(\bu)$ onto the tangent plane of $\calM$ at
$\bu$, taken in the discrete $L^2$ metric
$\langle\cdot,\cdot\rangle_h$. We call it the \emph{projected
gradient}. It coincides with the Riemannian gradient of $\bE_h$ on
$\calM$ for that metric, written $\nabla^{\calR}_h\bE_h(\bu)$ in
\cite{chen2024fully}, but we avoid that name here,  since  the scheme
\eqref{dys} is not a Riemannian gradient method. In contrast, the $H^1$ flow is the Riemannian gradient descent method using the
$H^1$ metric but needs to invert the shifted Laplacian twice per iteration
\cite{chen2024fully}.
\subsection{The scheme, the radial residual, and the Lyapunov
function}\label{ssec:scheme}

The metric matrix of the scheme is $\bbG_X:=\bbI-\gamma\Delta_h$, so
that $\bbG_X^{-1}=(\bbM+\gamma\bbS)^{-1}\bbM$ is the resolvent
appearing in the iteration. The associated $X$-inner product and
$X$-norm are
\begin{equation}
\langle\bv,\bw\rangle_X:=\langle\bv,\bbG_X^{-1}\bw\rangle_h,
\qquad
\norm{\bv}_X:=\langle\bv,\bv\rangle_X^{1/2}.
\label{d2:Xnorm}
\end{equation}
The operator $\bbG_X^{-1}$  need not be  a symmetric matrix, but it is
$h$-self-adjoint in the sense of \S\ref{sec:setting}, since
$\bbM\bbG_X^{-1}=\bbM(\bbM+\gamma\bbS)^{-1}\bbM$ is symmetric.
All properties of the $X$-norm used below follow from
the spectral decomposition of $\bbG_X^{-1}$. Let $\mu_j\ge0$ be the
eigenvalues of $-\Delta_h$ for an $\langle\cdot,\cdot\rangle_h$-orthonormal
eigenbasis, in which $\bbG_X^{-1}$ has the eigenvalues
$1/(1+\gamma\mu_j)\in(0,1]$, and let $\hat v_j$ be the coefficients
of $\bv$ in that basis. Then
\[
\norm{\bv}_X^2=\sum_j\frac{\hat v_j^2}{1+\gamma\mu_j},
\qquad
\norm{\bbG_X^{-1}\bv}_2^2=\sum_j\frac{\hat v_j^2}{(1+\gamma\mu_j)^2},
\]
and $0<(1+\gamma\mu_j)^{-2}\le(1+\gamma\mu_j)^{-1}\le1$ and
$\frac{\gamma\mu_j}{1+\gamma\mu_j}\le\gamma\mu_j$ imply
\begin{gather}
\bv\neq0\ \Longrightarrow\
0<\norm{\bv}_X\le\norm{\bv}_2,
\qquad
\norm{\bbG_X^{-1}\bv}_2\le\norm{\bv}_X ,
\label{d2:Gbounds}\\
\bu\in\calM\ \Longrightarrow\
0\le1-\norm{\bu}_X^2=\sum_j\frac{\gamma\mu_j}{1+\gamma\mu_j}\hat u_j^2
\le\gamma\langle\bu,-\Delta_h\bu\rangle_h .
\label{n:resolventfacts}
\end{gather}
In particular $\langle\cdot,\cdot\rangle_X$ is an inner product.
We also use the norm of the metric matrix itself,
\begin{equation}
\norm{\bv}_{\calG}^2:=\langle\bv,\bbG_X\bv\rangle_h
=\norm{\bv}_2^2+\gamma\langle\bv,-\Delta_h\bv\rangle_h ,
\qquad\text{so that}\qquad
\norm{\bv}_2\le\norm{\bv}_{\calG},
\quad
\norm{\bbG_X\bv}_X=\norm{\bv}_{\calG} .
\label{n:Gnorm}
\end{equation}
The descent of one step is measured in $\norm{\cdot}_{\calG}$ in
Section~\ref{sec:onestep} and converted to $\norm{\cdot}_X$
in Proposition~\ref{prop:onestep}.

One step of the scheme is \eqref{dys} with the iteration index
dropped: it maps $\bz\neq0$ to $\bu=\bz/\norm{\bz}_2$, $\bxit$, and
$\bz^{+}=\bz+\bxit-\bu$.
The iteration is $\bz^{k+1}=(\bz^k)^{+}$ started from the normalized
positive initial vector $\bz^0=\bu^0$ with $\bu^0>0$ and
$\norm{\bu^0}_2=1$. For every $\bz\neq0$, with $\bu$ the normalized
iterate of \eqref{dys}, define the \emph{radial residual} and
the \emph{Lyapunov function}
\begin{equation}
R(\bz):=\norm{\bz}_2-1+\gamma\lambda(\bu),
\qquad
\calL(\bz):=\bE_h(\bu)+\frac{R(\bz)^2}{\gamma}.
\label{d2:orbit}
\end{equation}
 Note that $R=0$ is equivalent to
$\norm{\bz}_2=1-\gamma\lambda(\bu)$, which is the length of a fixed
point of \eqref{dys}, and $R$ carries no information on the
direction of $\bz$. The deviation in the tangential directions is
measured by $\bg(\bu)$, and at a fixed point both $R$ and $\bg(\bu)$
vanish, see \eqref{d2:step}.
Whenever $\bz^k\neq0$ we abbreviate
$\bu^k:=\bz^k/\norm{\bz^k}_2$,
$R^k:=R(\bz^k)$ and $\calL^k:=\calL(\bz^k)$, together with
$\bg^k:=\bg(\bu^k)$. The superscript $+$ marks the
same quantities after one step, at the next iteration index, e.g.,
$\bu^{+}:=\bz^{+}/\norm{\bz^{+}}_2$ is $\bu^{k+1}$ whenever
$\bz^{+}\neq0$, and $R^{+}:=R(\bz^{+})$.

\subsection{Data constants and the step-size threshold}\label{ssec:data}
 
The energy $\bE_h(\bu^0)$ and the eigenvalue functional
$\lambda(\bu^0)$ of the initial vector enter through the level
\begin{equation}
E_0:=\bE_h(\bu^0)+\frac{\lambda(\bu^0)^2}{\max\{1,128\,\bE_h(\bu^0)\}} .
\label{n:cap}
\end{equation}
Since $\tfrac\beta4\norm{\bu}_4^4\le\bE_h(\bu)$, the identity
$\lambda(\bu)=2\bE_h(\bu)+\tfrac\beta2\norm{\bu}_4^4$ implies
$\lambda(\bu^0)\le4\bE_h(\bu^0)$. Since
$\bE_h(\bu^0)/\max\{1,128\,\bE_h(\bu^0)\}\le\tfrac1{128}$ in both
cases of the maximum, we have
\[
\frac{\lambda(\bu^0)^2}{\max\{1,128\,\bE_h(\bu^0)\}}
\le\frac{16\,\bE_h(\bu^0)^2}{\max\{1,128\,\bE_h(\bu^0)\}}
\le\frac{16}{128}\,\bE_h(\bu^0)=\frac18\,\bE_h(\bu^0),
\]
and hence $\bE_h(\bu^0)\le E_0\le\tfrac98\bE_h(\bu^0)$. We will
show that $\calL(\bz^k)\le E_0$ for every $k$.
Let $\Lambda_{\mathrm d}:=\max_i\bbS_{ii}/w_i$ be the
stiffness-to-mass ratio of the discrete Laplacian. Note that
$\Lambda_{\mathrm d}>0$, since $\bbS_{ii}$ is the value of the
quadratic form of $\bbS$ at the $i$th unit vector and $\bbS$ is
positive definite, see Appendix~\ref{app:fem}. With
$w_{\min}=\min_iw_i$ as in \S\ref{ssec:norms}, set
\begin{gather}
P_0:=C_S^2\bigl(2E_0+1\bigr),
\qquad
U_0:=\min\Bigl\{\frac1{w_{\min}},\ \frac{P_0}{w_{\min}^{1/3}},\
2\sqrt{\frac{E_0}{\beta\,w_{\min}}}\Bigr\},
\label{n:Pb}\\
B_0:=\norm{V}_\infty+\beta\min\{P_0^{3/2},U_0\},
\qquad
D:=E_0+B_0,
\qquad
B_1:=\norm{V}_\infty+\beta U_0,
\label{n:BD}\\
C_0:=3\beta C_SP_0,
\qquad
L_0:=\norm{V}_\infty+\min\bigl\{3\beta U_0,\
C_0+2\max\{C_0-\tfrac14,0\}^2\bigr\}.
\label{n:LH}
\end{gather}
Here $P_0$ bounds $\norm{\bu}_6^2$ and $U_0$ bounds
$\norm{\bu}_\infty^2$, while $B_0$, $B_1$ and $L_0$
bound the explicit part
\begin{equation}
\bn(\bu):=\bbV\bu+\beta\bu^3,
\qquad
\nabla\bE_h(\bu)=-\Delta_h\bu+\bn(\bu),
\label{n:ndef}
\end{equation}
of the gradient and its Jacobian, for normalized vectors of
energy at most $E_0$, see Lemmas~\ref{lem:ebd0} and
\ref{lem:nonlinear}. The step-size threshold of
Theorem~\ref{thm:main} is
\begin{equation}
\gamma_{\max}:=\min\Bigl\{1,\ \frac1{128D},\ \frac1{8L_0},\
\frac{2}{B_1+\sqrt{B_1^2+12D\Lambda_{\mathrm d}}}\Bigr\}.
\label{n:gammamax}
\end{equation}
Every entry is an explicit function of the data. The proof uses
$\gamma_{\max}$ only through the four conditions of the next lemma.

\begin{lemma}[The step-size conditions]\label{lem:conditions}
A step size $\gamma>0$ satisfies $\gamma\le\gamma_{\max}$ if and only
if
\begin{equation}
\text{(C1)}\ \ \gamma\le1,\qquad
\text{(C2)}\ \ \gamma D\le\tfrac1{128},\qquad
\text{(C3)}\ \ \gamma L_0\le\tfrac18,\qquad
\text{(C4)}\ \ \gamma B_1+3\gamma^2D\Lambda_{\mathrm d}\le1 .
\label{n:cond}
\end{equation}
Under these conditions $\gamma\le1/\max\{1,128\,\bE_h(\bu^0)\}$, and
therefore $\calL(\bz^0)=\bE_h(\bu^0)+\gamma\lambda(\bu^0)^2\le E_0$
for $\bz^0=\bu^0$.
\end{lemma}

\begin{proof}
(C1), (C2) and (C3) are the first three entries of
\eqref{n:gammamax}. The quadratic
$3D\Lambda_{\mathrm d}\gamma^2+B_1\gamma-1$ in (C4) has a
positive leading coefficient and the value $-1$ at $\gamma=0$, thus
it has one negative and one positive root, and it is nonpositive for
$\gamma>0$ if and only if $\gamma$ is at most the positive root
\[
\frac{-B_1+\sqrt{B_1^2+12D\Lambda_{\mathrm d}}}{6D\Lambda_{\mathrm d}}
=\frac{2}{B_1+\sqrt{B_1^2+12D\Lambda_{\mathrm d}}},
\]
which is the last entry of \eqref{n:gammamax}. Since
$D\ge E_0\ge\bE_h(\bu^0)$, (C1) and (C2) imply that $\gamma$ is at
most $1/\max\{1,128\,\bE_h(\bu^0)\}$. For $\bz^0=\bu^0$ the
residual is $R(\bz^0)=\gamma\lambda(\bu^0)$, which implies
$\calL(\bz^0)=\bE_h(\bu^0)+\gamma\lambda(\bu^0)^2\le
\bE_h(\bu^0)+\lambda(\bu^0)^2/\max\{1,128\,\bE_h(\bu^0)\}=E_0$.
\end{proof}

The next lemma collects the bounds implied by the energy level
$E_0$, which are used throughout Section~\ref{sec:onestep}.

\begin{lemma}[Consequences of the energy bound]\label{lem:ebd0}
Let $\bv\in\calM$ satisfy $\bE_h(\bv)\le E_0$. Then the following
hold.
\begin{itemize}
\item[(i)] $\langle\bv,-\Delta_h\bv\rangle_h\le2E_0$,
$\norm{\bv}_{1,h}^2\le2E_0+1$ and $\norm{\bv}_6^2\le P_0$.
\item[(ii)] $\norm{\bv}_\infty^2\le U_0$.
\item[(iii)] $\norm{\bv^3}_2\le\min\{P_0^{3/2},U_0\}$,
$\norm{\bv}_4^4\le\min\{P_0^{3/2},U_0\}$,
$\norm{\bn(\bv)}_2\le B_0$ and $\langle\bn(\bv),\bv\rangle_h\le B_0$.
\item[(iv)] $0\le\lambda(\bv)\le2D$, and thus
$0\le\gamma\lambda(\bv)\le2\gamma D$.
\end{itemize}
\end{lemma}

\begin{proof}
(i) Each of the three terms of $\bE_h$ in \eqref{fd-energy} is
nonnegative under Assumption~\ref{asp:struct}, which implies
\[
\tfrac12\langle\bv,-\Delta_h\bv\rangle_h\le\bE_h(\bv)\le E_0,
\qquad
\norm{\bv}_{1,h}^2=\langle\bv,-\Delta_h\bv\rangle_h+\norm{\bv}_2^2
\le2E_0+1,
\]
and the Sobolev embedding \eqref{sobolev} implies
$\norm{\bv}_6^2\le C_S^2(2E_0+1)=P_0$.

(ii) A single term of a sum of nonnegative numbers is at most the
whole sum. Applied to $\norm{\bv}_2^2=\sum_iw_iv_i^2=1$, to
$\norm{\bv}_6^6=\sum_iw_iv_i^6\le P_0^3$ and to
$\tfrac\beta4\norm{\bv}_4^4=\tfrac\beta4\sum_iw_iv_i^4\le E_0$,
this implies $v_i^2\le w_i^{-1}$, $v_i^2\le P_0w_i^{-1/3}$ and
$v_i^2\le2\sqrt{E_0/(\beta w_i)}$ for every $i$, and $w_i\ge w_{\min}$.

(iii) $\norm{\bv^3}_2=\norm{\bv}_6^3\le P_0^{3/2}$, and
$\norm{\bv^3}_2\le\norm{\bv}_\infty^2\norm{\bv}_2\le U_0$. The
Cauchy--Schwarz inequality applied to the pair $(|\bv|,|\bv|^3)$ in
$\langle\cdot,\cdot\rangle_h$ implies
$\norm{\bv}_4^4\le\norm{\bv}_2\norm{\bv^3}_2$. Since
$\norm{\bbV\bv}_2\le\norm{V}_\infty\norm{\bv}_2$ and
$\langle\bbV\bv,\bv\rangle_h\le\norm{V}_\infty$, both
$\norm{\bn(\bv)}_2$ and $\langle\bn(\bv),\bv\rangle_h
=\langle\bbV\bv,\bv\rangle_h+\beta\norm{\bv}_4^4$ are at most
$\norm{V}_\infty+\beta\min\{P_0^{3/2},U_0\}=B_0$.

(iv) $\lambda(\bv)=\langle\bv,-\Delta_h\bv\rangle_h
+\langle\bn(\bv),\bv\rangle_h$ is a sum of nonnegative terms, and
by (i) and (iii) it is at most $2E_0+B_0\le2D$.
\end{proof}

\begin{remark}[Mesh dependence and the size of the threshold]\label{d2:rem:mesh}
The mesh is fixed throughout and every constant above is a finite
number determined by that mesh and the data, and no statement below
involves a limit $h\to0$. To see how the threshold behaves under
refinement, consider a family of quasi-uniform meshes, so that
$w_{\min}$ is of order $h^d$ and $\Lambda_{\mathrm d}$ of order
$h^{-2}$, and a family of initial vectors whose energies $\bE_h(\bu^0)$ are
bounded independently of $h$, as are the normalized interpolants of
a fixed smooth function that vanishes on $\partial\Omega$. Then the
constants split into three groups.
The first consists of universal numbers. The second consists of
$E_0$, $P_0$, $B_0$, $D$ and $C_0$, which are built from
$\bE_h(\bu^0)$, from $\norm{V}_\infty$ and $\beta$, and from the constant
$C_S$ of the Sobolev embedding \eqref{sobolev}, and are therefore
bounded independently of $h$. The third consists of the pointwise
bound $U_0$, of order at most $h^{-d/3}$ by its middle alternative, of
$B_1$, which is built from it, and of $\Lambda_{\mathrm d}$.
Since $L_0\le\norm{V}_\infty+C_0+2\max\{C_0-\tfrac14,0\}^2$ belongs
to the second group,
the mesh restricts $\gamma_{\max}$ only through the last entry of
\eqref{n:gammamax}, which is of order $h$ for $d\le3$. The
assumption on the initial energy cannot be dropped: the normalized
constant vector used as the
initial guess in Section~\ref{sec:numerics} vanishes on the
boundary, and this jump gives it an energy $\bE_h(\bu^0)$ of order $h^{-1}$
on the finite difference grid, for which $\gamma_{\max}$ is of order
$h^{2}$ in three dimensions. On the other hand, for a fixed
mesh, potential and initial vector, the constants $E_0$, $B_0$, $D$,
$B_1$ and $L_0$ grow at most linearly in $\beta$, thus
$\gamma_{\max}$ is of order $\beta^{-1}$ as $\beta\to\infty$, which
is also the order of the step size rule used in
Section~\ref{sec:numerics}.
\end{remark}

\subsection{The main theorem}\label{sec:main}

\begin{theorem}[Global convergence to the discrete ground
state]\label{thm:main}
Let Assumption~\ref{asp:struct} hold, start the iteration
\eqref{dys} from a positive normalized vector,
$\bz^0=\bu^0>0$ with $\norm{\bu^0}_2=1$, and let
$0<\gamma\le\gamma_{\max}$. Here $\gamma_{\max}>0$ is a
threshold, given in \eqref{n:gammamax}, determined by
the initial vector $\bu^0$, by the data $V$
and $\beta$, by the constant $C_S$ of the Sobolev embedding
\eqref{sobolev}, by
the mesh size through $w_{\min}=\min_iw_i$, and by the
stiffness-to-mass ratio $\Lambda_{\mathrm d}$ of the discrete
Laplacian. Then $\bu^k>0$
entrywise for every $k\ge0$, and
\[
\bu^k\to\bu_{\mathrm{GS}},
\qquad
\bE_h(\bu^k)\to\bE_h(\bu_{\mathrm{GS}})=\min_{\calM}\bE_h ,
\]
where $\bu_{\mathrm{GS}}>0$ is the positive global minimizer of
$\bE_h$ on $\calM$.
\end{theorem}
 
\begin{remark}[The four step-size conditions]\label{d2:rem:constants}
The four conditions of Lemma~\ref{lem:conditions} enter the proof at
separate places. (C2) makes the radial deviation
$1-\norm{\bz^k}_2$, the residual $R^k$ and the change of direction
in one step small, of order $\gamma D$ or $\sqrt{\gamma D}$.
(C3) controls the second-order Taylor remainder of the explicit part
of the energy, by whichever of the pointwise and the Sobolev
estimates is better. (C4) is the positivity condition, and for the
families of bounded initial energy in Remark~\ref{d2:rem:mesh} it is
the only condition that restricts the step size through the mesh.
(C1) is used only in Lemma~\ref{lem:conditions}, to place the
initial value $\calL^0$ below $E_0$. The descent inequality \eqref{n:drop} of
Proposition~\ref{prop:onestep} has the universal coefficients
$\tfrac14$ and $\tfrac12$.
\end{remark}

\begin{remark}[The variant DYS~II]\label{d2:rem:dys2}
The analysis below applies without change to \eqref{dys2}, with
$V_1\ge0$ and $V_2\ge0$, after the substitutions
\[
-\Delta_h\ \to\ -\Delta_h+\bbV_1,
\qquad
\bn(\bu)\ \to\ \bbV_2\bu+\beta\bu^3,
\qquad
\norm{V}_\infty\ \to\ \norm{V_2}_\infty,
\qquad
\Lambda_{\mathrm d}\ \to\ \max_i\bigl(\bbS_{ii}/w_i+V_{1,i}\bigr).
\]
The energy and the level $E_0$ are unchanged, the metric
matrix becomes $\bbI-\gamma\Delta_h+\gamma\bbV_1$, and the norms
$\norm{\cdot}_X$ and $\norm{\cdot}_{\calG}$ are built from it. The matrix
$\bbM+\gamma\bbS+\gamma\bbM\bbV_1$ is again an irreducible Stieltjes
matrix, hence Lemma~\ref{lem:disc}(ii) holds for it, and since
$V_1\ge0$ the quadratic form
$\langle\bv,(-\Delta_h+\bbV_1)\bv\rangle_h$ is still at most
$2\bE_h(\bv)$ and still dominates
$\langle\bv,-\Delta_h\bv\rangle_h$ in the Sobolev inequality
\eqref{sobolev}. With these substitutions every constant of
\S\ref{ssec:data}, every identity and every estimate below holds
for DYS~II.
\end{remark}

\subsection*{Proof strategy}
 
The proof consists of three steps. Steps 1 and 2 are in
Section~\ref{sec:onestep}, with Step 2 in
\S\ref{sec:positivity}, and Step 3 is in Section~\ref{sec:convergence}.

\smallskip\noindent
Step 1 is the analysis of one step, from $\bz^k$ to $\bz^{k+1}$.
Two exact identities express the
change of energy and the new residual after one step in terms of the
change of direction $\bu^{k+1}-\bu^k$, the two lengths
$\norm{\bz^k}_2$ and $\norm{\bz^{k+1}}_2$, and the Taylor remainder
of the explicit part of the energy. In the energy identity the
implicit kinetic term of the step has a positive coefficient, and in
the residual identity the leading kinetic variation of $\lambda$
cancels. Positivity of $\bz^{k+1}$ and the one-sided radial interval
$1-3\gamma D\le\norm{\bz^{k+1}}_2\le1$ follow from the entrywise
positive resolvent and the sign structure of the discretization,
using only the energy bound of $\bu^k$. If $\bu^{k+1}$ also has
energy at most $E_0$, the identities imply the conditional descent
$\gamma(\calL^k-\calL^{k+1})\ge\frac5{16}\norm{\bu^{k+1}-\bu^k}_{\calG}^2
+\frac34(R^k)^2$.

\smallskip\noindent
Step 2 is an energy barrier argument. The condition on the new energy
is removed by a continuation argument, in which the step size and
the radial deviation of the input are scaled by the same factor. The
scaled input
has a smaller residual and a Lyapunov value still at most $E_0$, and
a first contact of the energy of the new direction with the level
$E_0$ would contradict the conditional descent. Hence the descent
holds unconditionally, which is Proposition~\ref{prop:onestep}.

\smallskip\noindent
Step 3 is the propagation and the passage to the limit. A single
induction propagates $\bu^k>0$, the radial interval and
$\calL^k\le E_0$. Summing the one-step drop along the whole sequence gives
$\sum_k\gamma\norm{\bg^k}_X^2<\infty$ and
$\sum_k(R^k)^2/\gamma<\infty$, thus $\norm{\bg^k}_2\to0$ and
$R^k\to0$. Every
accumulation point of $(\bu^k)$ is therefore a critical point of
$\bE_h|_{\calM}$, and it is nonnegative since it is a limit of
positive iterates. Since the only nonnegative critical point is the positive global
minimizer $\bu_{\mathrm{GS}}$ (Remark~\ref{d2:rem:gs}), every
accumulation point equals $\bu_{\mathrm{GS}}$, and by compactness of
$\calM$ the whole sequence converges. This is
Theorem~\ref{thm:main}. Once the iterates are close to
$\bu_{\mathrm{GS}}$, the same descent inequality implies a linear
rate, which is Theorem~\ref{thm:rate}.

\smallskip\noindent
 Sections~\ref{sec:onestep} and \ref{sec:convergence}
are devoted to the proofs of Theorems~\ref{thm:main} and
\ref{thm:rate}.

\section{Estimates and properties of one iteration}
\label{sec:onestep}

For convenience, in this section we drop the iteration index $k$
and mark the next iteration by the superscript $+$: $\bz$, $\bu$ and
$\bxit$ stand for $\bz^k$, $\bu^k$ and $\bxit^k$, and $\bz^+$,
$\bu^+$, $R^+$ and $\calL^+$ for $\bz^{k+1}$, $\bu^{k+1}$, $R^{k+1}$
and $\calL^{k+1}$, where $\bz\neq0$ and
$\bu^+=\bz^+/\norm{\bz^+}_2$ if $\bz^+\neq0$. The change of
direction is $\bdelta:=\bu^+-\bu$. Since
$\norm{\bu}_2=\norm{\bu^+}_2=1$ and $\bu^+=\bz^+/\norm{\bz^+}_2$, we
have
$\norm{\bdelta}_2^2=\norm{\bu^+}_2^2-2\langle\bu,\bu^+\rangle_h+\norm{\bu}_2^2
=2-2\langle\bu,\bu^+\rangle_h$ and
$\langle\bu,\bdelta\rangle_h=\langle\bu,\bu^+\rangle_h-1$, i.e.,
\begin{equation}
\langle\bu,\bdelta\rangle_h=-\tfrac12\norm{\bdelta}_2^2,
\qquad
\norm{\bdelta}_2^2=2\bigl(1-\langle\bu,\bu^+\rangle_h\bigr)
=2\Bigl(1-\frac{\langle\bu,\bz^+\rangle_h}{\norm{\bz^+}_2}\Bigr).
\label{n:chord}
\end{equation}
Let $H(\bu):=\tfrac12\langle\bbV\bu,\bu\rangle_h+\tfrac\beta4\norm{\bu}_4^4$
be the explicit part of the energy, as in Appendix~\ref{app:dys}, so
that $\bE_h(\bu)=\tfrac12\langle\bu,-\Delta_h\bu\rangle_h+H(\bu)$
and the gradient of $H$ with respect to $\langle\cdot,\cdot\rangle_h$
is $\bn$ of \eqref{n:ndef}. The Taylor remainder of $H$ along
$\bdelta$ is
\begin{equation}
\begin{gathered}
\calT(\bu,\bdelta):=H(\bu^+)-H(\bu)-\langle\bn(\bu),\bdelta\rangle_h
=\int_0^1(1-\sigma)\,\langle\bn'(\bu+\sigma\bdelta)\bdelta,
\bdelta\rangle_h\,d\sigma,
\\
\bn'(\bw):=\bbV+3\beta\diag(\bw^2),
\end{gathered}
\label{n:remainder}
\end{equation}
by Taylor's formula with integral remainder for the polynomial
$\sigma\mapsto H(\bu+\sigma\bdelta)$, whose second derivative is
$\langle\bn'(\bu+\sigma\bdelta)\bdelta,\bdelta\rangle_h$. Since the
kinetic part of $\bE_h$ is quadratic, we have
$\tfrac12\langle\bu^+,-\Delta_h\bu^+\rangle_h
=\tfrac12\langle\bu,-\Delta_h\bu\rangle_h+\langle-\Delta_h\bu,\bdelta\rangle_h
+\tfrac12\langle\bdelta,-\Delta_h\bdelta\rangle_h$, and adding
\eqref{n:remainder}, we get the Taylor expansion of $\bE_h$,
\begin{equation}
\bE_h(\bu^+)-\bE_h(\bu)=\langle\nabla\bE_h(\bu),\bdelta\rangle_h
+\tfrac12\langle\bdelta,-\Delta_h\bdelta\rangle_h+\calT(\bu,\bdelta),
\label{n:taylorE}
\end{equation}
where $\nabla\bE_h(\bu)=-\Delta_h\bu+\bn(\bu)$ by \eqref{n:ndef}.

\subsection{Exact identities for one step}\label{ssec:identities}

\begin{lemma}[Step identities]\label{lem:step}
Let $\bz\neq0$. Then
\begin{equation}
\bbG_X\bz^{+}=\bu-\gamma\bn(\bu)-\gamma\bigl(1-\norm{\bz}_2\bigr)(-\Delta_h\bu),
\qquad
\bbG_X\bigl(\bz^{+}-\bz\bigr)=-\bigl(R\,\bu+\gamma\,\bg(\bu)\bigr).
\label{d2:step}
\end{equation}
If in addition $\bz^+\neq0$, then
\begin{gather}
\norm{\bz^+}_2\,\bbG_X\bdelta=\bigl(1-\norm{\bz^+}_2\bigr)\bu
-\gamma\nabla\bE_h(\bu)
-\gamma\bigl(\norm{\bz^+}_2-\norm{\bz}_2\bigr)(-\Delta_h\bu),
\label{n:stepdelta}\\
\bigl(1+\gamma\langle\bu,-\Delta_h\bu\rangle_h\bigr)
\bigl(\norm{\bz^+}_2-\norm{\bz}_2\bigr)
=-R-\norm{\bz^+}_2\Bigl(\gamma\langle-\Delta_h\bu,\bdelta\rangle_h
-\tfrac12\norm{\bdelta}_2^2\Bigr).
\label{n:radialid}
\end{gather}
\end{lemma}

\begin{proof}
By \eqref{dys}, we have $\bz^{+}-\bz=\bxit-\bu$ and
$\bbG_X\bxit=2\bu-\bz-\gamma\bn(\bu)$, which imply
\[
\bbG_X\bz^+=\bbG_X\bz+2\bu-\bz-\gamma\bn(\bu)-\bbG_X\bu
=(\bbG_X-\bbI)(\bz-\bu)+\bu-\gamma\bn(\bu),
\]
and, since $\bbG_X-\bbI=-\gamma\Delta_h$ and
$\bz-\bu=(\norm{\bz}_2-1)\bu$, we have
$(\bbG_X-\bbI)(\bz-\bu)=(\norm{\bz}_2-1)(-\gamma\Delta_h)\bu$, which
implies the first identity. Since $\bz=\norm{\bz}_2\bu$, we have
$\bbG_X\bz=\norm{\bz}_2\bbG_X\bu
=\norm{\bz}_2\bu+\gamma\norm{\bz}_2(-\Delta_h\bu)$, and subtracting
it from the first identity, we have
\begin{equation}
\bbG_X(\bz^+-\bz)=(1-\norm{\bz}_2)\bu-\gamma(-\Delta_h\bu)-\gamma\bn(\bu)
=(1-\norm{\bz}_2)\bu-\gamma\nabla\bE_h(\bu),
\label{n:stepgrad}
\end{equation}
and $\nabla\bE_h(\bu)=\lambda(\bu)\bu+\bg(\bu)$ by \eqref{d2:grad}
and the definition \eqref{d2:orbit} of $R$ imply the second
identity. For \eqref{n:stepdelta}, we have
\begin{equation}
\bz^+-\bz=\norm{\bz^+}_2\bu^+-\norm{\bz}_2\bu
=\norm{\bz^+}_2\bdelta+(\norm{\bz^+}_2-\norm{\bz}_2)\bu .
\label{n:zdecomp}
\end{equation}
Applying $\bbG_X$ to \eqref{n:zdecomp}, with
$\bbG_X\bu=\bu+\gamma(-\Delta_h\bu)$, and using \eqref{n:stepgrad},
we have
\[
\norm{\bz^+}_2\bbG_X\bdelta
+(\norm{\bz^+}_2-\norm{\bz}_2)\bigl(\bu+\gamma(-\Delta_h\bu)\bigr)
=\bbG_X(\bz^+-\bz)
=(1-\norm{\bz}_2)\bu-\gamma\nabla\bE_h(\bu),
\]
and moving the second term of the left side to the right side, with
$(1-\norm{\bz}_2)-(\norm{\bz^+}_2-\norm{\bz}_2)=1-\norm{\bz^+}_2$,
we get \eqref{n:stepdelta}. For
\eqref{n:radialid}, we take the $\langle\cdot,\cdot\rangle_h$-product
of the second identity of \eqref{d2:step} with $\bu$ and use
$\langle\bg(\bu),\bu\rangle_h=0$, the decomposition \eqref{n:zdecomp},
$\langle\bbG_X\bdelta,\bu\rangle_h=\langle\bu,\bdelta\rangle_h
+\gamma\langle-\Delta_h\bu,\bdelta\rangle_h$, \eqref{n:chord} and
$\langle\bu,\bbG_X\bu\rangle_h=1+\gamma\langle\bu,-\Delta_h\bu\rangle_h$.
\end{proof}

\begin{lemma}[Energy and residual identities]\label{lem:identities}
Let $\bz\neq0$ with $\bz^+\neq0$. Then
\begin{align}
\gamma\bigl(\bE_h(\bu)-\bE_h(\bu^+)\bigr)
&=\frac{\norm{\bz^+}_2+1}{2}\norm{\bdelta}_2^2
+\gamma\Bigl(\norm{\bz^+}_2-\frac12\Bigr)\langle\bdelta,-\Delta_h\bdelta\rangle_h
\notag\\
&\qquad
+\gamma\bigl(\norm{\bz^+}_2-\norm{\bz}_2\bigr)\langle-\Delta_h\bu,\bdelta\rangle_h
-\gamma\,\calT(\bu,\bdelta),
\label{n:energyid}\\
R^{+}&=-\tfrac12\norm{\bdelta}_2^2
-\gamma\bigl(\norm{\bz^+}_2-\norm{\bz}_2\bigr)\langle-\Delta_h\bu,\bu^+\rangle_h
\notag\\
&\qquad
-\gamma\bigl(\norm{\bz^+}_2-1\bigr)\langle-\Delta_h\bdelta,\bu^+\rangle_h
+\gamma\langle\bn(\bu^+)-\bn(\bu),\bu^+\rangle_h .
\label{n:residid}
\end{align}
\end{lemma}

\begin{proof}
Taking the $\langle\cdot,\cdot\rangle_h$-product of
\eqref{n:stepdelta} with $\bdelta$ and using \eqref{n:chord}, we have
\[
\norm{\bz^+}_2\bigl(\norm{\bdelta}_2^2
+\gamma\langle\bdelta,-\Delta_h\bdelta\rangle_h\bigr)
=-\frac{1-\norm{\bz^+}_2}{2}\norm{\bdelta}_2^2
-\gamma\langle\nabla\bE_h(\bu),\bdelta\rangle_h
-\gamma\bigl(\norm{\bz^+}_2-\norm{\bz}_2\bigr)
\langle-\Delta_h\bu,\bdelta\rangle_h .
\]
Eliminating $\gamma\langle\nabla\bE_h(\bu),\bdelta\rangle_h$ between
the equation above and \eqref{n:taylorE} multiplied by $\gamma$, we
obtain \eqref{n:energyid}. For \eqref{n:residid},
we take the $\langle\cdot,\cdot\rangle_h$-product of the first
identity of \eqref{d2:step} with $\bu^+$. Since
$\bz^+=\norm{\bz^+}_2\bu^+$, we have
$\langle\bbG_X\bz^+,\bu^+\rangle_h
=\norm{\bz^+}_2\langle\bbG_X\bu^+,\bu^+\rangle_h
=\norm{\bz^+}_2+\norm{\bz^+}_2\gamma\langle\bu^+,-\Delta_h\bu^+\rangle_h$,
while $\langle\bu,\bu^+\rangle_h=1-\frac12\norm{\bdelta}_2^2$, thus
\[
\norm{\bz^+}_2-1=-\tfrac12\norm{\bdelta}_2^2
+\gamma\bigl(\norm{\bz}_2-1\bigr)\langle-\Delta_h\bu,\bu^+\rangle_h
-\gamma\langle\bn(\bu),\bu^+\rangle_h
-\norm{\bz^+}_2\gamma\langle\bu^+,-\Delta_h\bu^+\rangle_h .
\]
Adding $\gamma\lambda(\bu^+)=\gamma\langle\bu^+,-\Delta_h\bu^+\rangle_h
+\gamma\langle\bn(\bu^+),\bu^+\rangle_h$, which follows from
\eqref{d2:grad} and \eqref{n:ndef}, to both sides, and writing
$\norm{\bz}_2-1=(\norm{\bz^+}_2-1)-(\norm{\bz^+}_2-\norm{\bz}_2)$ and
$\langle-\Delta_h\bu^+,\bu^+\rangle_h=\langle-\Delta_h\bu,\bu^+\rangle_h
+\langle-\Delta_h\bdelta,\bu^+\rangle_h$, we get \eqref{n:residid}.
\end{proof}

\subsection{Estimates for the explicit part}\label{ssec:explicit}

\begin{lemma}[The explicit part on an energy sublevel]\label{lem:nonlinear}
Let (C3) in \eqref{n:cond} hold, and assume $\bz\neq0$, $\bz^+\neq0$,
$\bE_h(\bu)\le E_0$ and $\bE_h(\bu^+)\le E_0$. Then
\begin{equation}
\gamma\,\calT(\bu,\bdelta)\le\tfrac1{16}\norm{\bdelta}_{\calG}^2,
\qquad
\gamma\bigl|\langle\bn(\bu^+)-\bn(\bu),\bu^+\rangle_h\bigr|
\le3\gamma B_0\norm{\bdelta}_2\le3\gamma D\,\norm{\bdelta}_{\calG} .
\label{n:nonlinear}
\end{equation}
\end{lemma}

\begin{proof}
For $0\le\sigma\le1$, the triangle inequality applied to
$\bu+\sigma\bdelta=(1-\sigma)\bu+\sigma\bu^+$ and Lemma~\ref{lem:ebd0}
at $\bu$ and $\bu^+$ give
$\norm{\bu+\sigma\bdelta}_\infty\le(1-\sigma)\norm{\bu}_\infty
+\sigma\norm{\bu^+}_\infty\le\sqrt{U_0}$ and, in the same way,
$\norm{\bu+\sigma\bdelta}_6\le\sqrt{P_0}$. By the definition of
$\bn'$ in \eqref{n:remainder}, it follows that
\[
\langle\bn'(\bu+\sigma\bdelta)\bdelta,\bdelta\rangle_h
=\sum_iw_iV_i\delta_i^2+3\beta\sum_iw_i(u_i+\sigma\delta_i)^2\delta_i^2
\le\norm{V}_\infty\norm{\bdelta}_2^2
+3\beta\sum_iw_i(u_i+\sigma\delta_i)^2\delta_i^2 .
\]
The last sum is bounded in two ways. First,
$\sum_iw_i(u_i+\sigma\delta_i)^2\delta_i^2\le U_0\norm{\bdelta}_2^2$,
which gives
\begin{equation}
\langle\bn'(\bu+\sigma\bdelta)\bdelta,\bdelta\rangle_h
\le(\norm{V}_\infty+3\beta U_0)\norm{\bdelta}_2^2 .
\label{n:pointwisebound}
\end{equation}
Second,
H\"older's inequality with exponents $3$ and $\frac32$ gives
$\sum_iw_i(u_i+\sigma\delta_i)^2\delta_i^2
\le\norm{\bu+\sigma\bdelta}_6^2\norm{\bdelta}_3^2
\le P_0\norm{\bdelta}_3^2$,
and H\"older's inequality with exponents $\frac43$ and $4$ gives
\[
\norm{\bdelta}_3^3=\sum_iw_i|\delta_i|^{3/2}|\delta_i|^{3/2}
\le\norm{\bdelta}_2^{3/2}\norm{\bdelta}_6^{3/2},
\qquad\text{i.e.,}\qquad
\norm{\bdelta}_3^2\le\norm{\bdelta}_2\norm{\bdelta}_6 .
\]
With the Sobolev embedding \eqref{sobolev},
$\norm{\bdelta}_6\le C_S\norm{\bdelta}_{1,h}$, we get
\begin{align}
\langle\bn'(\bu+\sigma\bdelta)\bdelta,\bdelta\rangle_h
&\le\norm{V}_\infty\norm{\bdelta}_2^2+C_0\norm{\bdelta}_2\norm{\bdelta}_{1,h}
\notag\\
&\le\bigl(\norm{V}_\infty+C_0+2\max\{C_0-\tfrac14,0\}^2\bigr)\norm{\bdelta}_2^2
+\tfrac18\langle\bdelta,-\Delta_h\bdelta\rangle_h .
\label{n:sobolevbound}
\end{align}
For the last step, let $\bdelta\neq0$ and
$a:=\norm{\bdelta}_{1,h}/\norm{\bdelta}_2\ge1$. Then
$\langle\bdelta,-\Delta_h\bdelta\rangle_h=(a^2-1)\norm{\bdelta}_2^2$
by the definition \eqref{sobolev} of $\norm{\cdot}_{1,h}$, and
$C_0\norm{\bdelta}_2\norm{\bdelta}_{1,h}
-\tfrac18\langle\bdelta,-\Delta_h\bdelta\rangle_h
=\bigl(C_0a-\tfrac{a^2-1}{8}\bigr)\norm{\bdelta}_2^2$. The concave
quadratic $a\mapsto C_0a-\frac{a^2-1}8$ attains its maximum over
$a\ge1$ at $a=\max\{4C_0,1\}$, where it equals
$C_0+2\max\{C_0-\frac14,0\}^2$. By \eqref{n:LH}, $L_0$ is the
smaller of the two numbers $\norm{V}_\infty+3\beta U_0$ and
$\norm{V}_\infty+C_0+2\max\{C_0-\tfrac14,0\}^2$, thus (C3), i.e.,
$\gamma L_0\le\frac18$, means that
$\gamma(\norm{V}_\infty+3\beta U_0)\le\frac18$ if
$3\beta U_0\le C_0+2\max\{C_0-\tfrac14,0\}^2$, and that
$\gamma\bigl(\norm{V}_\infty+C_0+2\max\{C_0-\tfrac14,0\}^2\bigr)\le\frac18$
otherwise.
In the first case \eqref{n:pointwisebound} multiplied by $\gamma$
gives
$\gamma\langle\bn'(\bu+\sigma\bdelta)\bdelta,\bdelta\rangle_h
\le\frac18\norm{\bdelta}_2^2$, and in the second case
\eqref{n:sobolevbound} multiplied by $\gamma$ gives
$\gamma\langle\bn'(\bu+\sigma\bdelta)\bdelta,\bdelta\rangle_h
\le\frac18\norm{\bdelta}_2^2
+\frac18\gamma\langle\bdelta,-\Delta_h\bdelta\rangle_h$. In both
cases
$\gamma\langle\bn'(\bu+\sigma\bdelta)\bdelta,\bdelta\rangle_h
\le\frac18\norm{\bdelta}_{\calG}^2$ by \eqref{n:Gnorm}.
Integrating against $(1-\sigma)$
in \eqref{n:remainder} halves this bound, which is the first
inequality. For the second bound, \eqref{n:ndef} and the
componentwise factorization
$(\bu^+)^3-\bu^3=\bigl((\bu^+)^2+\bu\bu^++\bu^2\bigr)\bdelta$ give
$\bn(\bu^+)-\bn(\bu)=\bbV\bdelta+\beta\bigl((\bu^+)^2+\bu\bu^++\bu^2\bigr)\bdelta$,
and since componentwise products commute,
\[
\langle\bn(\bu^+)-\bn(\bu),\bu^+\rangle_h
=\langle\bdelta,\bbV\bu^+\rangle_h
+\beta\langle\bdelta,(\bu^+)^3\rangle_h
+\beta\langle\bdelta,\bu(\bu^+)^2\rangle_h
+\beta\langle\bdelta,\bu^2\bu^+\rangle_h .
\]
By the Cauchy--Schwarz inequality and
$\norm{\bbV\bu^+}_2\le\norm{V}_\infty\norm{\bu^+}_2=\norm{V}_\infty$,
\[
\bigl|\langle\bn(\bu^+)-\bn(\bu),\bu^+\rangle_h\bigr|
\le\norm{\bdelta}_2\bigl(\norm{V}_\infty+\beta\norm{(\bu^+)^3}_2
+\beta\norm{\bu(\bu^+)^2}_2+\beta\norm{\bu^2\bu^+}_2\bigr).
\]
Each of the three mixed cubes is bounded in two ways. By H\"older's
inequality with exponents $3$ and $\frac32$ and Lemma~\ref{lem:ebd0}(i)
at $\bu$ and $\bu^+$,
\[
\norm{\bu(\bu^+)^2}_2^2=\sum_iw_iu_i^2(u_i^+)^4
\le\Bigl(\sum_iw_iu_i^6\Bigr)^{1/3}\Bigl(\sum_iw_i(u_i^+)^6\Bigr)^{2/3}
=\norm{\bu}_6^2\norm{\bu^+}_6^4\le P_0^3,
\]
and in the same way
$\norm{\bu^2\bu^+}_2^2\le\norm{\bu}_6^4\norm{\bu^+}_6^2\le P_0^3$ and
$\norm{(\bu^+)^3}_2^2=\norm{\bu^+}_6^6\le P_0^3$. By
Lemma~\ref{lem:ebd0}(ii) at $\bu$ and $\bu^+$ and
$\norm{\bu}_2=\norm{\bu^+}_2=1$,
\begin{gather*}
\norm{\bu(\bu^+)^2}_2\le\norm{\bu}_\infty\norm{\bu^+}_\infty\norm{\bu^+}_2\le U_0,
\qquad
\norm{\bu^2\bu^+}_2\le\norm{\bu}_\infty\norm{\bu^+}_\infty\norm{\bu}_2\le U_0,
\\
\norm{(\bu^+)^3}_2\le\norm{\bu^+}_\infty^2\norm{\bu^+}_2\le U_0.
\end{gather*}
Thus each mixed cube is at most $\min\{P_0^{3/2},U_0\}$, the bracket is
at most $\norm{V}_\infty+3\beta\min\{P_0^{3/2},U_0\}\le3B_0$, and the
second bound follows from $B_0\le D$ and
$\norm{\bdelta}_2\le\norm{\bdelta}_{\calG}$.
\end{proof}

\subsection{Positivity and the radial interval}\label{ssec:steppos}

\begin{lemma}[Positivity and the radial interval]\label{lem:posradius}
Let Assumption~\ref{asp:struct} and (C2), (C4) in \eqref{n:cond}
hold, and assume $\bz\neq0$ and
\begin{equation}
\bu>0,\qquad \bE_h(\bu)\le E_0,\qquad 1-3\gamma D\le\norm{\bz}_2\le1 .
\label{n:hyp}
\end{equation}
Then $\bz^+>0$ componentwise, and
\begin{equation}
1-3\gamma D\le\langle\bu,\bz^+\rangle_h\le\norm{\bz^+}_2\le1,
\qquad
\norm{\bdelta}_2^2\le2\bigl(1-\langle\bu,\bz^+\rangle_h\bigr)\le6\gamma D,
\qquad
|R|\le3\gamma D .
\label{n:radius}
\end{equation}
\end{lemma}

\begin{proof}
We first prove the positivity of $\bz^+$. By the first identity of
\eqref{d2:step},
$\bbG_X\bz^+=\bu-\gamma\bn(\bu)-\gamma(1-\norm{\bz}_2)(-\Delta_h\bu)$,
and we bound its components from below. By \eqref{n:ndef},
$n_i(\bu)=(V_i+\beta u_i^2)u_i$, and $V_i\le\norm{V}_\infty$,
$u_i^2\le\norm{\bu}_\infty^2\le U_0$ by Lemma~\ref{lem:ebd0}(ii) and
$u_i>0$ give $n_i(\bu)\le(\norm{V}_\infty+\beta U_0)u_i=B_1u_i$ by
\eqref{n:BD}. Since $-\Delta_h=\bbM^{-1}\bbS$ with
$\bbM=\diag\{w_1,\dots,w_N\}$, we have
$(-\Delta_h\bu)_i=w_i^{-1}\bigl(\bbS_{ii}u_i+\sum_{j\neq i}\bbS_{ij}u_j\bigr)
\le w_i^{-1}\bbS_{ii}u_i\le\Lambda_{\mathrm d}u_i$, since
$\bbS_{ij}\le0$ for $j\neq i$ by Lemma~\ref{lem:disc}(ii),
$\bu>0$, and $\Lambda_{\mathrm d}=\max_i\bbS_{ii}/w_i$. With
$0\le1-\norm{\bz}_2\le3\gamma D$ from \eqref{n:hyp}, we get
\[
(\bbG_X\bz^+)_i=u_i-\gamma n_i(\bu)-\gamma(1-\norm{\bz}_2)(-\Delta_h\bu)_i
\ \ge\ u_i-\gamma B_1u_i-3\gamma^2D\Lambda_{\mathrm d}u_i
=\bigl(1-\gamma B_1-3\gamma^2D\Lambda_{\mathrm d}\bigr)u_i
\ \ge\ 0
\]
for every $i$, by (C4), i.e., $\bbG_X\bz^+\ge0$. Since
$\bbG_X^{-1}=(\bbM+\gamma\bbS)^{-1}\bbM$ has strictly positive entries
by Lemma~\ref{lem:disc}(ii) and $w_i>0$, and a matrix with strictly
positive entries maps every nonnegative nonzero vector to a positive
vector, $\bz^+=\bbG_X^{-1}(\bbG_X\bz^+)>0$ whenever $\bz^+\neq0$.
Suppose that $\bz^+=0$. Taking the
$\langle\cdot,\cdot\rangle_h$-product of the first identity of
\eqref{d2:step} with $\bu$ and using $\norm{\bu}_2=1$,
$\langle\bn(\bu),\bu\rangle_h\le B_0\le D$ by Lemma~\ref{lem:ebd0}(iii),
$1-\norm{\bz}_2\le3\gamma D$ and
$\langle\bu,-\Delta_h\bu\rangle_h\le2E_0\le2D$ by Lemma~\ref{lem:ebd0}(i),
we get
\[
0=\langle\bu,\bbG_X\bz^+\rangle_h=1-\gamma\langle\bn(\bu),\bu\rangle_h
-\gamma(1-\norm{\bz}_2)\langle\bu,-\Delta_h\bu\rangle_h
\ge1-\gamma D-6\gamma^2D^2>0,
\]
where the last step holds since $\gamma D\le\frac1{128}$ by (C2), a
contradiction. Thus $\bz^+\neq0$ and $\bz^+>0$.

Next we prove the upper bound $\norm{\bz^+}_2\le1$. By \eqref{dys},
$\bz^+=\bz-\bu+\bxit$ with
$\bxit=\bbG_X^{-1}\bigl(2\bu-\bz-\gamma\bn(\bu)\bigr)$, and
$\bz=\norm{\bz}_2\bu$ gives $\bz-\bu=(\norm{\bz}_2-1)\bu$ and
$2\bu-\bz=(2-\norm{\bz}_2)\bu$, thus
\begin{equation}
\bz^+=(\norm{\bz}_2-1)\bu+(2-\norm{\bz}_2)\bbG_X^{-1}\bu
-\gamma\bbG_X^{-1}\bn(\bu).
\label{n:zplus}
\end{equation}
Since $V_i\ge0$ and $\bu>0$, we have $\bn(\bu)\ge0$ by
\eqref{n:ndef}, and $\bbG_X^{-1}>0$ entrywise as shown above, thus
$\bbG_X^{-1}\bn(\bu)\ge0$ and
$0<\bz^+\le\bigl((\norm{\bz}_2-1)\bbI+(2-\norm{\bz}_2)\bbG_X^{-1}\bigr)\bu$
componentwise. Since the weighted norm $\norm{\cdot}_2$ is monotone on
nonnegative vectors, $\norm{\bz^+}_2$ is at most the norm of the
right side. In the $\langle\cdot,\cdot\rangle_h$-orthonormal
eigenbasis of \S\ref{ssec:scheme}, in which $\bbG_X^{-1}$ has the
eigenvalues $1/(1+\gamma\mu_j)\in(0,1]$, the operator in parentheses
has the eigenvalues
$(\norm{\bz}_2-1)+(2-\norm{\bz}_2)/(1+\gamma\mu_j)$, which lie in
$[\norm{\bz}_2-1,1]\subset[-1,1]$ for $0<\norm{\bz}_2\le1$, since
$2-\norm{\bz}_2>0$. Thus, with $\hat u_j$ the coefficients of $\bu$
in that basis,
\[
\norm{\bz^+}_2^2
\le\sum_j\Bigl((\norm{\bz}_2-1)+\frac{2-\norm{\bz}_2}{1+\gamma\mu_j}\Bigr)^2\hat u_j^2
\le\sum_j\hat u_j^2=\norm{\bu}_2^2=1 .
\]

Finally we prove the lower bound. Taking the
$\langle\cdot,\cdot\rangle_h$-product of \eqref{n:zplus} with $\bu$,
and using $\norm{\bu}_2=1$,
$\langle\bu,\bbG_X^{-1}\bu\rangle_h=\norm{\bu}_X^2$ by
\eqref{d2:Xnorm} and the $h$-self-adjointness of $\bbG_X^{-1}$ in the
last term, we have
$\langle\bu,\bz^+\rangle_h
=(\norm{\bz}_2-1)+(2-\norm{\bz}_2)\norm{\bu}_X^2
-\gamma\langle\bbG_X^{-1}\bu,\bn(\bu)\rangle_h$, and since
$1-(\norm{\bz}_2-1)=2-\norm{\bz}_2$,
\[
1-\langle\bu,\bz^+\rangle_h=(2-\norm{\bz}_2)\bigl(1-\norm{\bu}_X^2\bigr)
+\gamma\langle\bbG_X^{-1}\bu,\bn(\bu)\rangle_h .
\]
Both terms are nonnegative, the first by \eqref{n:resolventfacts}
and the second since $\bbG_X^{-1}\bu>0$ and $\bn(\bu)\ge0$. By
\eqref{n:resolventfacts} and Lemma~\ref{lem:ebd0}(i), we have
$1-\norm{\bu}_X^2\le\gamma\langle\bu,-\Delta_h\bu\rangle_h\le2\gamma E_0$,
and by the Cauchy--Schwarz inequality, \eqref{d2:Gbounds} and
Lemma~\ref{lem:ebd0}(iii),
$\gamma\langle\bbG_X^{-1}\bu,\bn(\bu)\rangle_h
\le\gamma\norm{\bbG_X^{-1}\bu}_2\norm{\bn(\bu)}_2
\le\gamma\norm{\bu}_X\norm{\bn(\bu)}_2\le\gamma B_0$. Therefore, with
$2-\norm{\bz}_2=1+(1-\norm{\bz}_2)$, $1-\norm{\bz}_2\le3\gamma D$,
$2E_0+B_0\le2D$ and $E_0\le D$,
\[
0\le1-\langle\bu,\bz^+\rangle_h
\le\gamma(2E_0+B_0)+2\gamma E_0\,(1-\norm{\bz}_2)
\le2\gamma D+6\gamma^2D^2\le3\gamma D ,
\]
where the last step holds since $6\gamma D\le\frac{6}{128}<1$ by
(C2). This proves the first chain in \eqref{n:radius}, since
$\langle\bu,\bz^+\rangle_h\le\norm{\bu}_2\norm{\bz^+}_2=\norm{\bz^+}_2$
by the Cauchy--Schwarz inequality. Since
$\langle\bu,\bz^+\rangle_h\ge1-3\gamma D>0$ and $\norm{\bz^+}_2\le1$,
we have
$\langle\bu,\bz^+\rangle_h/\norm{\bz^+}_2\ge\langle\bu,\bz^+\rangle_h$,
and \eqref{n:chord} implies
$\norm{\bdelta}_2^2=2\bigl(1-\langle\bu,\bz^+\rangle_h/\norm{\bz^+}_2\bigr)
\le2\bigl(1-\langle\bu,\bz^+\rangle_h\bigr)\le6\gamma D$. The bound
on $R$ follows from the definition \eqref{d2:orbit},
$R=-(1-\norm{\bz}_2)+\gamma\lambda(\bu)$, in which
$-(1-\norm{\bz}_2)\in[-3\gamma D,0]$ by \eqref{n:hyp} and
$\gamma\lambda(\bu)\in[0,2\gamma D]$ by Lemma~\ref{lem:ebd0}(iv), thus
$R\in[-3\gamma D,2\gamma D]$.
\end{proof}

\subsection{Conditional descent}\label{ssec:conddescent}

\begin{lemma}[Conditional descent]\label{lem:conddescent}
Let the hypotheses of Lemma~\ref{lem:posradius} hold together with
(C3), and assume in addition that $\bE_h(\bu^+)\le E_0$. Then
\begin{equation}
\gamma\bigl(\calL(\bz)-\calL(\bz^{+})\bigr)\ \ge\
\tfrac5{16}\,\norm{\bdelta}_{\calG}^2+\tfrac34\,R^2 ,
\label{n:conddrop}
\end{equation}
and the right side vanishes only if $R=0$ and $\bg(\bu)=0$, in
which case $\bz^+=\bz$.
\end{lemma}

\begin{proof}
Since $\bE_h(\bu)\le E_0$ and $\bE_h(\bu^+)\le E_0$,
Lemma~\ref{lem:ebd0}(i) at $\bu$ and $\bu^+$ and $E_0\le D$ give
\begin{equation}
\gamma\langle\bu,-\Delta_h\bu\rangle_h\le2\gamma E_0\le2\gamma D,
\qquad
\gamma\langle\bu^+,-\Delta_h\bu^+\rangle_h\le2\gamma D,
\label{n:kinbounds}
\end{equation}
which are used throughout the proof. By \eqref{discreteL2norm} and
$\Delta_h=-\bbM^{-1}\bbS$, we have
$\langle\bv,-\Delta_h\bw\rangle_h=\bv^\top\bbM\bbM^{-1}\bbS\bw=\bv^\top\bbS\bw$
for all $\bv,\bw$, and $\bbS$ is symmetric positive semidefinite by
Lemma~\ref{lem:disc}(ii), thus
$(\bv,\bw)\mapsto\langle\bv,-\Delta_h\bw\rangle_h$ is a symmetric
positive semidefinite bilinear form, for which the Cauchy--Schwarz
inequality
\[
|\langle\bv,-\Delta_h\bw\rangle_h|
\le\sqrt{\langle\bv,-\Delta_h\bv\rangle_h}\sqrt{\langle\bw,-\Delta_h\bw\rangle_h}
\]
holds. With \eqref{n:kinbounds} and
$\gamma\langle\bdelta,-\Delta_h\bdelta\rangle_h\le\norm{\bdelta}_{\calG}^2$
from \eqref{n:Gnorm}, we have
\begin{equation}
\begin{gathered}
\gamma\bigl|\langle-\Delta_h\bu,\bdelta\rangle_h\bigr|
\le\sqrt{\gamma\langle\bu,-\Delta_h\bu\rangle_h}\,
\sqrt{\gamma\langle\bdelta,-\Delta_h\bdelta\rangle_h}
\le\sqrt{2\gamma D}\,\norm{\bdelta}_{\calG},
\\
\gamma\bigl|\langle-\Delta_h\bdelta,\bu^+\rangle_h\bigr|
\le\sqrt{\gamma\langle\bu^+,-\Delta_h\bu^+\rangle_h}\,
\sqrt{\gamma\langle\bdelta,-\Delta_h\bdelta\rangle_h}
\le\sqrt{2\gamma D}\,\norm{\bdelta}_{\calG},
\\
\gamma\bigl|\langle-\Delta_h\bu,\bu^+\rangle_h\bigr|
\le\sqrt{\gamma\langle\bu,-\Delta_h\bu\rangle_h}\,
\sqrt{\gamma\langle\bu^+,-\Delta_h\bu^+\rangle_h}
\le2\gamma D .
\end{gathered}
\label{n:csbounds}
\end{equation}
By \eqref{n:radius} and $\norm{\bdelta}_2\le\norm{\bdelta}_{\calG}$
from \eqref{n:Gnorm}, we have
\begin{equation}
\tfrac12\norm{\bdelta}_2^2\le\tfrac12\norm{\bdelta}_2\sqrt{6\gamma D}
\le\sqrt{3\gamma D/2}\,\norm{\bdelta}_{\calG},
\qquad
0\le1-\norm{\bz^+}_2\le3\gamma D,
\qquad
\norm{\bz^+}_2\le1 .
\label{n:smallfacts}
\end{equation}
Taking absolute values in \eqref{n:radialid}, with
$1+\gamma\langle\bu,-\Delta_h\bu\rangle_h\ge1$, and using
\eqref{n:csbounds} and \eqref{n:smallfacts}, we get
\begin{align}
\bigl|\norm{\bz^+}_2-\norm{\bz}_2\bigr|
&\le|R|+\norm{\bz^+}_2\Bigl(\gamma\bigl|\langle-\Delta_h\bu,\bdelta\rangle_h\bigr|
+\tfrac12\norm{\bdelta}_2^2\Bigr)
\notag\\
&\le|R|+\bigl(\sqrt2+\sqrt{3/2}\bigr)\sqrt{\gamma D}\,\norm{\bdelta}_{\calG}
\le|R|+\tfrac83\sqrt{\gamma D}\,\norm{\bdelta}_{\calG} ,
\label{n:ebound}
\end{align}
since $\sqrt2+\sqrt{3/2}<2.64<\frac83$.

We first estimate the energy by \eqref{n:energyid}. Since
$\norm{\bz^+}_2\ge1-3\gamma D$ by \eqref{n:radius} and
$\norm{\bz^+}_2\le1$, the coefficients of the first two terms on
the right side of \eqref{n:energyid} satisfy
$\frac{\norm{\bz^+}_2+1}{2}\ge\norm{\bz^+}_2-\frac12\ge\frac12-3\gamma D>0$,
and since $\norm{\bdelta}_2^2\ge0$ and
$\gamma\langle\bdelta,-\Delta_h\bdelta\rangle_h\ge0$ add up to
$\norm{\bdelta}_{\calG}^2$ by \eqref{n:Gnorm}, the sum of these two
terms is at least $(\frac12-3\gamma D)\norm{\bdelta}_{\calG}^2$. By
\eqref{n:csbounds} and \eqref{n:ebound}, the third term on the right
side of \eqref{n:energyid} is at least
\[
-\bigl|\norm{\bz^+}_2-\norm{\bz}_2\bigr|\,\gamma\bigl|\langle-\Delta_h\bu,\bdelta\rangle_h\bigr|
\ge-\bigl|\norm{\bz^+}_2-\norm{\bz}_2\bigr|\sqrt{2\gamma D}\,\norm{\bdelta}_{\calG}
\ge-\sqrt{2\gamma D}\,|R|\,\norm{\bdelta}_{\calG}-\tfrac{8\sqrt2}{3}\gamma D\,\norm{\bdelta}_{\calG}^2,
\]
and Young's inequality $xy\le\frac15x^2+\frac54y^2$ with $x=|R|$ and
$y=\sqrt{2\gamma D}\,\norm{\bdelta}_{\calG}$ gives
$\sqrt{2\gamma D}\,|R|\,\norm{\bdelta}_{\calG}\le\frac15R^2+\frac52\gamma D\,\norm{\bdelta}_{\calG}^2$.
With Lemma~\ref{lem:nonlinear} for the remainder term,
$-\gamma\calT(\bu,\bdelta)\ge-\frac1{16}\norm{\bdelta}_{\calG}^2$, we
get
\begin{equation}
\gamma\bigl(\bE_h(\bu)-\bE_h(\bu^+)\bigr)
\ge\Bigl(\tfrac12-\tfrac1{16}
-\bigl(3+\tfrac{8\sqrt2}3+\tfrac52\bigr)\gamma D\Bigr)\norm{\bdelta}_{\calG}^2
-\tfrac15R^2
\ge\Bigl(\tfrac7{16}-10\gamma D\Bigr)\norm{\bdelta}_{\calG}^2-\tfrac15R^2 ,
\label{n:energydrop}
\end{equation}
since $3+\frac{8\sqrt2}{3}+\frac52<9.28<10$.

Next we bound the four terms on the right side of \eqref{n:residid}
in absolute value. The first term is bounded by
$\sqrt{3\gamma D/2}\,\norm{\bdelta}_{\calG}$ via \eqref{n:smallfacts},
the second term by
$2\gamma D\bigl|\norm{\bz^+}_2-\norm{\bz}_2\bigr|
\le2\gamma D|R|+\frac{16}3(\gamma D)^{3/2}\norm{\bdelta}_{\calG}$
via \eqref{n:csbounds} and \eqref{n:ebound}, the third term by
$3\gamma D\cdot\sqrt{2\gamma D}\,\norm{\bdelta}_{\calG}$ via
\eqref{n:smallfacts} and \eqref{n:csbounds}, and the fourth term by
$3\gamma D\,\norm{\bdelta}_{\calG}$ via Lemma~\ref{lem:nonlinear}.
Adding these bounds, we have
\[
|R^+|\le2\gamma D|R|
+\Bigl(\sqrt{\tfrac32}+\bigl(\tfrac{16}3+3\sqrt2\bigr)\gamma D
+3\sqrt{\gamma D}\Bigr)\sqrt{\gamma D}\,\norm{\bdelta}_{\calG}
\le2\gamma D|R|+\tfrac85\sqrt{\gamma D}\,\norm{\bdelta}_{\calG},
\]
since the bracket is increasing in $\gamma D$ and bounded by
$1.2248+0.0749+0.2652<1.57<\frac85$ at $\gamma D=\frac1{128}$, the
largest value allowed by (C2). Squaring with
$(x+y)^2\le2x^2+2y^2$, we obtain
\begin{equation}
(R^+)^2\le2(2\gamma D)^2R^2+2\bigl(\tfrac85\bigr)^2\gamma D\,\norm{\bdelta}_{\calG}^2
=8\gamma^2D^2R^2+\tfrac{128}{25}\,\gamma D\,\norm{\bdelta}_{\calG}^2 .
\label{n:residdrop}
\end{equation}

By the definition \eqref{d2:orbit} of $\calL$, we have
$\gamma(\calL(\bz)-\calL(\bz^+))=\gamma(\bE_h(\bu)-\bE_h(\bu^+))
+R^2-(R^+)^2$, thus adding \eqref{n:energydrop} and
\eqref{n:residdrop} we get
\[
\gamma\bigl(\calL(\bz)-\calL(\bz^+)\bigr)
\ge\Bigl(\tfrac7{16}-\bigl(10+\tfrac{128}{25}\bigr)\gamma D\Bigr)\norm{\bdelta}_{\calG}^2
+\bigl(\tfrac45-8\gamma^2D^2\bigr)R^2
\ge\tfrac5{16}\norm{\bdelta}_{\calG}^2+\tfrac34R^2,
\]
since $(10+\frac{128}{25})\gamma D\le16\gamma D\le\frac18$ and
$8\gamma^2D^2\le\frac{8}{128^2}<\frac1{20}$ by (C2), and
$\frac7{16}-\frac18=\frac5{16}$, $\frac45-\frac1{20}=\frac34$.

It remains to prove the equality case. If the right side of
\eqref{n:conddrop} vanishes, then $\norm{\bdelta}_{\calG}=0$ and
$R=0$. Since $\norm{\cdot}_{\calG}$ is a norm by \eqref{n:Gnorm},
$\bdelta=0$, i.e., $\bu^+=\bu$, and \eqref{n:zdecomp} becomes
$\bz^+-\bz=(\norm{\bz^+}_2-\norm{\bz}_2)\bu$. With $R=0$, the second
identity of \eqref{d2:step} then becomes
$(\norm{\bz^+}_2-\norm{\bz}_2)\,\bbG_X\bu=-\gamma\bg(\bu)$. Taking its
$\langle\cdot,\cdot\rangle_h$-product with $\bu$ and using
$\langle\bg(\bu),\bu\rangle_h=0$, we get
$(\norm{\bz^+}_2-\norm{\bz}_2)\langle\bu,\bbG_X\bu\rangle_h=0$, and
since
$\langle\bu,\bbG_X\bu\rangle_h=\norm{\bu}_{\calG}^2\ge\norm{\bu}_2^2=1$
by \eqref{n:Gnorm}, we have $\norm{\bz^+}_2=\norm{\bz}_2$, thus
$\bz^+=\bz$ and $\bg(\bu)=0$.
\end{proof}

\begin{lemma}[The projected gradient and the increment]\label{lem:gradform}
Under the hypotheses of Lemma~\ref{lem:posradius},
\begin{equation}
\gamma^2\norm{\bg(\bu)}_X^2\le\tfrac54\norm{\bdelta}_{\calG}^2+10\gamma D\,R^2,
\qquad
\norm{\bz^+-\bz}_{\calG}^2\le \norm{\bdelta}_{\calG}^2+R^2 .
\label{n:gradform}
\end{equation}
\end{lemma}

\begin{proof}
Let
\[
\bdelta_\perp:=\bdelta
-\frac{\langle\bu,\bbG_X\bdelta\rangle_h}{\langle\bu,\bbG_X\bu\rangle_h}\,\bu ,
\qquad\text{so that}\qquad
\langle\bu,\bbG_X\bdelta_\perp\rangle_h=0,
\qquad
\norm{\bdelta}_{\calG}^2=\norm{\bdelta_\perp}_{\calG}^2
+\frac{\langle\bu,\bbG_X\bdelta\rangle_h^2}{\langle\bu,\bbG_X\bu\rangle_h}
\ge\norm{\bdelta_\perp}_{\calG}^2 ,
\]
i.e., $\bdelta_\perp$ is the projection of $\bdelta$ onto the
orthogonal complement of $\bu$ with respect to the inner product
$\langle\cdot,\bbG_X\cdot\rangle_h$ of the norm $\norm{\cdot}_{\calG}$.
By \eqref{n:chord} and the $h$-self-adjointness of $-\Delta_h$, we
have $\langle\bu,\bbG_X\bdelta\rangle_h=\langle\bu,\bdelta\rangle_h
+\gamma\langle\bu,-\Delta_h\bdelta\rangle_h
=-\tfrac12\norm{\bdelta}_2^2+\gamma\langle-\Delta_h\bu,\bdelta\rangle_h$
and $\langle\bu,\bbG_X\bu\rangle_h=1+\gamma\langle\bu,-\Delta_h\bu\rangle_h$,
thus \eqref{n:radialid} reads
$\langle\bu,\bbG_X\bu\rangle_h\bigl(\norm{\bz^+}_2-\norm{\bz}_2\bigr)
=-R-\norm{\bz^+}_2\langle\bu,\bbG_X\bdelta\rangle_h$. Substituting
this and
$\bdelta=\bdelta_\perp+\frac{\langle\bu,\bbG_X\bdelta\rangle_h}{\langle\bu,\bbG_X\bu\rangle_h}\bu$
into \eqref{n:zdecomp}, the terms with
$\langle\bu,\bbG_X\bdelta\rangle_h$ cancel and we have
\begin{equation}
\bz^+-\bz=\norm{\bz^+}_2\bdelta+(\norm{\bz^+}_2-\norm{\bz}_2)\bu
=\norm{\bz^+}_2\,\bdelta_\perp
-\frac{R}{\langle\bu,\bbG_X\bu\rangle_h}\,\bu ,
\label{n:zdecompperp}
\end{equation}
and substituting \eqref{n:zdecompperp} into the second identity of
\eqref{d2:step}, we have
\begin{equation}
\gamma\bg(\bu)=-\bbG_X(\bz^+-\bz)-R\bu
=-\norm{\bz^+}_2\,\bbG_X\bdelta_\perp
-R\Bigl(\bu-\frac{\bbG_X\bu}{\langle\bu,\bbG_X\bu\rangle_h}\Bigr).
\label{n:gdecomp}
\end{equation}
The two terms on the right side of \eqref{n:zdecompperp} are
orthogonal with respect to $\langle\cdot,\bbG_X\cdot\rangle_h$, since
$\langle\bu,\bbG_X\bdelta_\perp\rangle_h=0$, and
$\norm{\bu}_{\calG}^2=\langle\bu,\bbG_X\bu\rangle_h$ by \eqref{n:Gnorm},
thus
\[
\norm{\bz^+-\bz}_{\calG}^2
=\norm{\bz^+}_2^2\norm{\bdelta_\perp}_{\calG}^2
+\frac{R^2}{\langle\bu,\bbG_X\bu\rangle_h^2}\norm{\bu}_{\calG}^2
=\norm{\bz^+}_2^2\norm{\bdelta_\perp}_{\calG}^2
+\frac{R^2}{\langle\bu,\bbG_X\bu\rangle_h}
\le \norm{\bdelta}_{\calG}^2+R^2 ,
\]
by $\norm{\bz^+}_2\le1$ from \eqref{n:radius},
$\norm{\bdelta_\perp}_{\calG}\le\norm{\bdelta}_{\calG}$ and
$\langle\bu,\bbG_X\bu\rangle_h=1+\gamma\langle\bu,-\Delta_h\bu\rangle_h\ge1$,
which is the second bound in \eqref{n:gradform}. For the first
bound, we take the $\norm{\cdot}_X$-norm of \eqref{n:gdecomp}. We
have $\norm{\bbG_X\bdelta_\perp}_X=\norm{\bdelta_\perp}_{\calG}\le \norm{\bdelta}_{\calG}$ by
\eqref{n:Gnorm}. By \eqref{d2:Xnorm} and \eqref{n:Gnorm},
$\langle\bu,\bbG_X\bu\rangle_X=\langle\bu,\bbG_X^{-1}\bbG_X\bu\rangle_h=\norm{\bu}_2^2=1$
and $\norm{\bbG_X\bu}_X^2=\norm{\bu}_{\calG}^2=\langle\bu,\bbG_X\bu\rangle_h$,
thus, expanding the square and using $\norm{\bu}_X\le\norm{\bu}_2=1$
from \eqref{d2:Gbounds}, Lemma~\ref{lem:ebd0}(i) and $E_0\le D$,
\begin{align*}
\norm{\bu-\frac{\bbG_X\bu}{\langle\bu,\bbG_X\bu\rangle_h}}_X^2
&=\norm{\bu}_X^2-\frac{2\langle\bu,\bbG_X\bu\rangle_X}{\langle\bu,\bbG_X\bu\rangle_h}
+\frac{\norm{\bbG_X\bu}_X^2}{\langle\bu,\bbG_X\bu\rangle_h^2}
=\norm{\bu}_X^2-\frac{1}{\langle\bu,\bbG_X\bu\rangle_h}
\\
&\le1-\frac{1}{1+\gamma\langle\bu,-\Delta_h\bu\rangle_h}
=\frac{\gamma\langle\bu,-\Delta_h\bu\rangle_h}{1+\gamma\langle\bu,-\Delta_h\bu\rangle_h}
\le\gamma\langle\bu,-\Delta_h\bu\rangle_h\le2\gamma D .
\end{align*}
By the triangle inequality in \eqref{n:gdecomp}, we get
$\gamma\norm{\bg(\bu)}_X\le\norm{\bz^+}_2\norm{\bdelta}_{\calG}+\sqrt{2\gamma D}\,|R|
\le \norm{\bdelta}_{\calG}+\sqrt{2\gamma D}\,|R|$, and
$(x+y)^2\le\frac54x^2+5y^2$ implies the first bound.
\end{proof}

\subsection{The energy barrier and the one-step estimates}\label{sec:positivity}

Lemma~\ref{lem:conddescent} assumes that the new direction has energy
at most $E_0$. In this subsection we remove this assumption by a
continuation argument, in which the step size and the radial
deviation of the input are scaled by the same factor.

\begin{lemma}[Energy barrier]\label{lem:barrier}
Let Assumption~\ref{asp:struct} and (C2)--(C4) in \eqref{n:cond}
hold, and assume $\bz\neq0$, $\bu>0$,
$1-3\gamma D\le\norm{\bz}_2\le1$ and $\calL(\bz)\le E_0$. Then
$\bE_h(\bu^+)\le E_0$.
\end{lemma}

\begin{proof}
For $0\le\theta\le1$, set
\begin{equation}
\bz_\theta:=\bigl(1-\theta(1-\norm{\bz}_2)\bigr)\bu .
\label{n:continuation}
\end{equation}
Since $0\le1-\norm{\bz}_2\le3\gamma D<1$ by the hypothesis
$1-3\gamma D\le\norm{\bz}_2\le1$ and (C2), the factor
$1-\theta(1-\norm{\bz}_2)$ in \eqref{n:continuation} lies in
$[1-3\theta\gamma D,1]\subset(0,1]$, and $\norm{\bu}_2=1$ gives
\begin{equation}
\norm{\bz_\theta}_2=1-\theta(1-\norm{\bz}_2)\in[1-3\theta\gamma D,1],
\qquad
1-\norm{\bz_\theta}_2=\theta(1-\norm{\bz}_2),
\label{n:thetalength}
\end{equation}
thus the normalized vector of $\bz_\theta$ is $\bu$. Let
$\bz_\theta^+$ denote the result of one step \eqref{dys} of size
$\theta\gamma$ from $\bz_\theta$, i.e., by the first identity of
\eqref{d2:step} with $\theta\gamma$ in place of $\gamma$ and
\eqref{n:thetalength},
\begin{equation}
\bz_\theta^+=(\bbI-\theta\gamma\Delta_h)^{-1}
\bigl(\bu-\theta\gamma\bn(\bu)
-\theta^2\gamma(1-\norm{\bz}_2)(-\Delta_h\bu)\bigr),
\qquad \bz_0^+=\bu,\qquad \bz_1^+=\bz^+ .
\label{n:zalpha}
\end{equation}
Quantities belonging to the step size $\theta\gamma$ have the
subscript $\theta$: $R_\theta(\cdot)$ and $\calL_\theta(\cdot)$ are
the residual and the Lyapunov function of \eqref{d2:orbit} with
$\gamma$ replaced by $\theta\gamma$, and
$\bu_\theta^+:=\bz_\theta^+/\norm{\bz_\theta^+}_2$ if
$\bz_\theta^+\neq0$. Since $\bz_1^+=\bz^+$, it suffices to prove
that $\bE_h(\bu_\theta^+)\le E_0$ for every $0\le\theta\le1$.

We first show that the continuation is admissible. For $0<\theta\le1$
we have $\theta\gamma\le\gamma$, and the left sides of
(C2)--(C4) in \eqref{n:cond} are increasing functions of $\gamma$,
since the constants $D$, $L_0$, $B_1$ and $\Lambda_{\mathrm d}$ of
\S\ref{ssec:data} do not depend on $\gamma$, thus (C2)--(C4) hold
with $\theta\gamma$ in place of $\gamma$. With
\eqref{n:thetalength}, $\bu>0$ and
$\bE_h(\bu)\le\calL(\bz)\le E_0$ by \eqref{d2:orbit}, $\bz_\theta$
satisfies \eqref{n:hyp} with $\theta\gamma$ in place of $\gamma$. Then Lemma~\ref{lem:posradius}, applied to $\bz_\theta$ with the
step size $\theta\gamma$, implies $\bz_\theta^+>0$ and
$\norm{\bz_\theta^+}_2\ge1-3\theta\gamma D>0$, and $\bu_\theta^+$
is defined for every $\theta\in[0,1]$, with $\bu_0^+=\bu$ by
\eqref{n:zalpha}. Since
$(\bbI-\theta\gamma\Delta_h)^{-1}=(\bbM+\theta\gamma\bbS)^{-1}\bbM$ and
$\bbM+\theta\gamma\bbS$ is positive definite for $\theta\ge0$, the
entries of $\bz_\theta^+$ in \eqref{n:zalpha} are rational functions
of $\theta$ whose denominators do not vanish on $[0,1]$, thus
$\theta\mapsto\bE_h(\bu_\theta^+)$ is continuous on $[0,1]$ and
differentiable at $\theta=0$. By \eqref{d2:orbit} with $\theta\gamma$
in place of $\gamma$ and \eqref{n:thetalength}, for $0<\theta\le1$
the residual and the Lyapunov function of $\bz_\theta$ are
\begin{equation}
\begin{gathered}
R_\theta(\bz_\theta)=\norm{\bz_\theta}_2-1+\theta\gamma\lambda(\bu)
=-\theta(1-\norm{\bz}_2)+\theta\gamma\lambda(\bu)=\theta R,
\\
\calL_\theta(\bz_\theta)=\bE_h(\bu)+\frac{(\theta R)^2}{\theta\gamma}
=\bE_h(\bu)+\frac{\theta R^2}{\gamma}
\le\bE_h(\bu)+\frac{R^2}{\gamma}=\calL(\bz)\le E_0 ,
\end{gathered}
\label{n:scaling}
\end{equation}
where we used $\bu=\bz_\theta/\norm{\bz_\theta}_2$ in the
computation of the Lyapunov function.

Next we consider the case $R=0$ and $\bg(\bu)=0$. The
second identity of \eqref{d2:step}, applied to the step of size
$\theta\gamma$ from $\bz_\theta$, whose normalized vector is $\bu$ and
whose residual is $\theta R$ by \eqref{n:scaling}, implies
$(\bbI-\theta\gamma\Delta_h)(\bz_\theta^+-\bz_\theta)
=-\bigl(\theta R\bu+\theta\gamma\bg(\bu)\bigr)
=-\theta\bigl(R\bu+\gamma\bg(\bu)\bigr)=0$, thus
$\bz_\theta^+=\bz_\theta=\norm{\bz_\theta}_2\bu$, which implies
$\bu_\theta^+=\bu$ and $\bE_h(\bu_\theta^+)=\bE_h(\bu)\le E_0$ for
every $\theta$. In the rest of the proof we assume that $R\neq0$ or
$\bg(\bu)\neq0$.

We claim that there is $\theta_1>0$ such that
$\bE_h(\bu_\theta^+)<E_0$ for $0<\theta<\theta_1$. If
$\bE_h(\bu)<E_0$, this follows from the continuity of
$\theta\mapsto\bE_h(\bu_\theta^+)$ at $\theta=0$, where its value is
$\bE_h(\bu_0^+)=\bE_h(\bu)$. If $\bE_h(\bu)=E_0$, then $R=0$ since
$\calL(\bz)=\bE_h(\bu)+R^2/\gamma\le E_0$ by \eqref{d2:orbit}, and
therefore $\bg(\bu)\neq0$. Multiplying \eqref{n:zalpha} by $\bbI-\theta\gamma\Delta_h$ gives
$(\bbI-\theta\gamma\Delta_h)\bz_\theta^+
=\bu-\theta\gamma\bn(\bu)-\theta^2\gamma(1-\norm{\bz}_2)(-\Delta_h\bu)$,
and differentiating both sides with respect to $\theta$ by the
product rule gives
\[
-\gamma\Delta_h\bz_\theta^+
+(\bbI-\theta\gamma\Delta_h)\frac{d}{d\theta}\bz_\theta^+
=-\gamma\bn(\bu)-2\theta\gamma(1-\norm{\bz}_2)(-\Delta_h\bu).
\]
At $\theta=0$, where $\bz_0^+=\bu$, this reads
$-\gamma\Delta_h\bu+\frac{d}{d\theta}\bz_\theta^+\big|_{\theta=0}
=-\gamma\bn(\bu)$, thus, by \eqref{n:ndef},
\[
\frac{d}{d\theta}\bz_\theta^+\Big|_{\theta=0}
=\gamma\Delta_h\bu-\gamma\bn(\bu)
=-\gamma\bigl(-\Delta_h\bu+\bn(\bu)\bigr)=-\gamma\nabla\bE_h(\bu).
\]
Differentiating $\norm{\bz_\theta^+}_2^2=\langle\bz_\theta^+,\bz_\theta^+\rangle_h$
gives
$\frac{d}{d\theta}\norm{\bz_\theta^+}_2
=\langle\bz_\theta^+,\frac{d}{d\theta}\bz_\theta^+\rangle_h/\norm{\bz_\theta^+}_2$,
thus the quotient rule applied to
$\bu_\theta^+=\bz_\theta^+/\norm{\bz_\theta^+}_2$ gives
\[
\frac{d}{d\theta}\bu_\theta^+
=\frac{1}{\norm{\bz_\theta^+}_2}\frac{d}{d\theta}\bz_\theta^+
-\frac{\langle\bz_\theta^+,\frac{d}{d\theta}\bz_\theta^+\rangle_h}
{\norm{\bz_\theta^+}_2^{3}}\,\bz_\theta^+ .
\]
At $\theta=0$, where $\bz_0^+=\bu$, $\norm{\bu}_2=1$ and
$\frac{d}{d\theta}\bz_\theta^+\big|_{\theta=0}=-\gamma\nabla\bE_h(\bu)$,
this gives, by the definitions \eqref{d2:grad} of $\lambda(\bu)$ and
$\bg(\bu)$,
\[
\frac{d}{d\theta}\bu_\theta^+\Big|_{\theta=0}
=-\gamma\nabla\bE_h(\bu)+\gamma\langle\bu,\nabla\bE_h(\bu)\rangle_h\,\bu
=-\gamma\bigl(\nabla\bE_h(\bu)-\lambda(\bu)\bu\bigr)=-\gamma\bg(\bu),
\]
and then, by the chain rule,
\[
\frac{d}{d\theta}\bE_h(\bu_\theta^+)\Big|_{\theta=0}
=\bigl\langle\nabla\bE_h(\bu),-\gamma\bg(\bu)\bigr\rangle_h
=-\gamma\bigl\langle\lambda(\bu)\bu+\bg(\bu),\bg(\bu)\bigr\rangle_h
=-\gamma\norm{\bg(\bu)}_2^2<0,
\]
where $\nabla\bE_h(\bu)=\lambda(\bu)\bu+\bg(\bu)$ by \eqref{d2:grad},
$\langle\bg(\bu),\bu\rangle_h=0$ and $\bg(\bu)\neq0$ are used in
the last two steps. A function with value $E_0$ and negative
derivative at $\theta=0$ is below $E_0$ on some interval
$(0,\theta_1)$, which proves the claim.

Suppose that $\bE_h(\bu_\theta^+)>E_0$ for some $\theta\in(0,1]$, and
let $\theta_*$ be the infimum of all such $\theta$. By the claim,
every such $\theta$ is at least $\theta_1$, thus
$\theta_*\ge\theta_1>0$. By the definition of the infimum,
$\bE_h(\bu_\theta^+)\le E_0$ for $0<\theta<\theta_*$, and by
continuity $\bE_h(\bu_{\theta_*}^+)\le E_0$ as a limit from the left
and $\bE_h(\bu_{\theta_*}^+)\ge E_0$ as a limit of values above $E_0$,
thus $\bE_h(\bu_{\theta_*}^+)=E_0$. Then Lemma~\ref{lem:conddescent}
applies to the step of size $\theta_*\gamma$ from $\bz_{\theta_*}$,
whose new normalized vector is $\bu_{\theta_*}^+$, since
(C2)--(C4) hold with $\theta_*\gamma$ in place of $\gamma$,
$\bz_{\theta_*}$ satisfies \eqref{n:hyp} with $\theta_*\gamma$ in
place of $\gamma$, as shown at the beginning of the proof, and
$\bE_h(\bu_{\theta_*}^+)=E_0$. By \eqref{d2:orbit},
\eqref{n:conddrop} and \eqref{n:scaling}, we get
\[
E_0=\bE_h(\bu_{\theta_*}^+)\le\calL_{\theta_*}(\bz_{\theta_*}^+)
\le\calL_{\theta_*}(\bz_{\theta_*})\le E_0 .
\]
Therefore equality holds throughout, in particular
$\calL_{\theta_*}(\bz_{\theta_*})-\calL_{\theta_*}(\bz_{\theta_*}^+)=0$,
thus the right side of \eqref{n:conddrop} vanishes for that step, and
the equality case of Lemma~\ref{lem:conddescent} gives
$R_{\theta_*}(\bz_{\theta_*})=\theta_*R=0$ and $\bg(\bu)=0$. Since
$\theta_*>0$, we get $R=0$ and $\bg(\bu)=0$, which contradicts the
assumption that $R\neq0$ or $\bg(\bu)\neq0$. Thus
$\bE_h(\bu_\theta^+)\le E_0$ for every $\theta\in[0,1]$.
\end{proof}

\begin{proposition}[One-step estimates]\label{prop:onestep}
Let Assumption~\ref{asp:struct} hold, let $0<\gamma\le\gamma_{\max}$,
and assume $\bz\neq0$, $\bu>0$, $1-3\gamma D\le\norm{\bz}_2\le1$
and $\calL(\bz)\le E_0$. Then $\bz^+>0$ componentwise,
$1-3\gamma D\le\norm{\bz^+}_2\le1$, $\bE_h(\bu^+)\le E_0$, and
\begin{equation}
\calL(\bz)-\calL(\bz^{+})
\ \ge\ \tfrac14\,\gamma\norm{\bg(\bu)}_X^2+\frac{R^2}{2\gamma},
\qquad
\calL(\bz)-\calL(\bz^{+})
\ \ge\ \frac{5}{16\gamma}\,\norm{\bz^{+}-\bz}_{\calG}^2 .
\label{n:drop}
\end{equation}
\end{proposition}

\begin{proof}
By Lemma~\ref{lem:conditions}, (C1)--(C4) hold. By
\eqref{d2:orbit}, $\bE_h(\bu)\le\calL(\bz)\le E_0$, thus $\bz$
satisfies \eqref{n:hyp}, and Lemma~\ref{lem:posradius} implies
$\bz^+>0$ and $1-3\gamma D\le\norm{\bz^+}_2\le1$,
Lemma~\ref{lem:barrier} implies $\bE_h(\bu^+)\le E_0$, and then
Lemma~\ref{lem:conddescent} implies
$\gamma(\calL(\bz)-\calL(\bz^+))\ge\frac5{16}\norm{\bdelta}_{\calG}^2+\frac34R^2$.
The first bound in \eqref{n:gradform} multiplied by $\frac14$ gives
$\frac5{16}\norm{\bdelta}_{\calG}^2\ge\frac14\gamma^2\norm{\bg(\bu)}_X^2
-\frac52\gamma D\,R^2$, thus
\[
\gamma\bigl(\calL(\bz)-\calL(\bz^+)\bigr)
\ge\tfrac14\gamma^2\norm{\bg(\bu)}_X^2+\bigl(\tfrac34-\tfrac52\gamma D\bigr)R^2
\ge\tfrac14\gamma^2\norm{\bg(\bu)}_X^2+\tfrac12R^2,
\]
since $\frac52\gamma D\le\frac{5}{256}<\frac14$ by (C2), and dividing
by $\gamma$ proves the first inequality in \eqref{n:drop}. The
second follows from
$\frac5{16}\norm{\bdelta}_{\calG}^2+\frac34R^2\ge\frac5{16}(\norm{\bdelta}_{\calG}^2+R^2)
\ge\frac5{16}\norm{\bz^+-\bz}_{\calG}^2$, by the second
bound in \eqref{n:gradform}, divided by $\gamma$.
\end{proof}

\section{Global convergence to the ground state}
\label{sec:convergence}

We now prove Theorem~\ref{thm:main} by induction, and then prove the
local linear rate. In particular, the induction is used to prove the
following invariance:
\begin{equation}
\bu^k>0,\qquad 1-3\gamma D\le\norm{\bz^k}_2\le1,\qquad
\calL^k\le E_0 .
\label{d2:invariant}
\end{equation}
The second clause contains $\bz^k\neq0$, so that $\bu^k$, $R^k$ and
$\calL^k$ are defined, and the third clause contains the energy bound
$\bE_h(\bu^k)\le E_0$. The second and third clauses imply, by
$R^k=\norm{\bz^k}_2-1+\gamma\lambda(\bu^k)$ and
Lemma~\ref{lem:ebd0}(iv), the residual bound $|R^k|\le3\gamma D$.

\begin{proof}[Proof of Theorem~\ref{thm:main}]
We first prove by induction that \eqref{d2:invariant}
holds at every step $k\ge0$, together with the drop
\begin{equation}
\calL^k-\calL^{k+1}
\ \ge\ \tfrac14\,\gamma\norm{\bg(\bu^k)}_X^2+\frac{(R^k)^2}{2\gamma},
\qquad
\calL^k-\calL^{k+1}
\ \ge\ \frac{5}{16\gamma}\norm{\bz^{k+1}-\bz^k}_{\calG}^2 .
\label{d2:drop}
\end{equation}
At $k=0$, $\bu^0>0$ and $\norm{\bz^0}_2=1$ by hypothesis, and
$\calL^0\le E_0$ by Lemma~\ref{lem:conditions}, thus
\eqref{d2:invariant} holds at step $0$. Suppose that
\eqref{d2:invariant} holds at step $k$. Then $\bz^k$ satisfies the
hypotheses of Proposition~\ref{prop:onestep}, which implies
$\bz^{k+1}>0$, thus $\bu^{k+1}=\bz^{k+1}/\norm{\bz^{k+1}}_2>0$, the
radial interval $1-3\gamma D\le\norm{\bz^{k+1}}_2\le1$, and the drop
\eqref{d2:drop} at step $k$, whose right sides are nonnegative, thus
$\calL^{k+1}\le\calL^k\le E_0$. Hence \eqref{d2:invariant} holds at
step $k+1$. By induction, \eqref{d2:invariant} and \eqref{d2:drop}
hold at every step $k\ge0$. In particular $\bu^k>0$ and $\bz^k\neq0$
for every $k\ge0$, so the whole sequence is well defined.

Next we prove the summability of the descent terms.
Summing \eqref{d2:drop} over $0\le k\le m$, the left side
equals $\calL^0-\calL^{m+1}\le\calL^0$, since
$\calL^{m+1}\ge\bE_h(\bu^{m+1})\ge0$, every term of $\bE_h$ being
nonnegative under Assumption~\ref{asp:struct}. Thus for every $m$ we
have
\[
\tfrac14\sum_{k=0}^m\gamma\norm{\bg^k}_X^2
+\frac12\sum_{k=0}^m\frac{(R^k)^2}{\gamma}\ \le\ \calL^0 ,
\qquad
\frac{5}{16\gamma}\sum_{k=0}^m\norm{\bz^{k+1}-\bz^k}_{\calG}^2
\ \le\ \calL^0 ,
\]
and letting $m\to\infty$, we get
\[
\sum_{k\ge0}\gamma\norm{\bg^k}_X^2\le4\,\calL^0,
\qquad
\sum_{k\ge0}\frac{(R^k)^2}{\gamma}\le2\,\calL^0,
\qquad
\sum_{k\ge0}\norm{\bz^{k+1}-\bz^k}_{\calG}^2\le\tfrac{16}{5}\,\gamma\,\calL^0 .
\]
The terms of these convergent series tend to zero and $\gamma>0$ is
fixed, thus $\norm{\bg^k}_X\to0$ and $R^k\to0$. With $\mu_{\max}$
the largest eigenvalue of $-\Delta_h$, the spectral representation
of $\norm{\cdot}_X$ in \S\ref{ssec:scheme} gives, for every
$\bv\in\bR^N$,
\begin{equation}
\norm{\bv}_2^2=\sum_j\hat v_j^2
\le(1+\gamma\mu_{\max})\sum_j\frac{\hat v_j^2}{1+\gamma\mu_j}
=(1+\gamma\mu_{\max})\norm{\bv}_X^2 ,
\label{n:Xequiv}
\end{equation}
thus also $\norm{\bg^k}_2\to0$.

Next we consider the accumulation points. The manifold
$\calM$ is a closed bounded
subset of $\bR^N$, thus compact, so $(\bu^k)$ has accumulation
points. Let $\bu^{k_j}\to\bu^*$. Each component of
$\bg(\bu)=\nabla\bE_h(\bu)-\lambda(\bu)\bu$ in
\eqref{d2:grad} is a
polynomial in the components of $\bu$, so $\bg$ is
continuous and $\bg(\bu^{k_j})
\to\bg(\bu^*)$, while
$\norm{\bg(\bu^k)}_2\to0$ by the previous
paragraph. Thus
$\bg(\bu^*)=0$, so $\bu^*$ is a critical point of
$\bE_h|_{\calM}$. By \eqref{d2:invariant} we have
$\bu^k>0$ for every $k\ge0$, and passing to the limit along
$\bu^{k_j}\to\bu^*$ we get $\bu^*\ge0$.
The strict positivity of $\bu^*$ will follow from
Remark~\ref{d2:rem:gs}.

Finally we identify the limit. Every accumulation point $\bu^*$
is thus a nonnegative critical point of $\bE_h|_{\calM}$, and by
Remark~\ref{d2:rem:gs} the only such point is $\bu_{\mathrm{GS}}$.
Thus $\bu^*=\bu_{\mathrm{GS}}$ for every accumulation point.
Suppose that $\bu^k\not\to\bu_{\mathrm{GS}}$. Then there
are $\epsilon>0$ and a subsequence $(\bu^{k_j})$ with
$\norm{\bu^{k_j}-\bu_{\mathrm{GS}}}_2\ge\epsilon$ for every $j$, which
has an accumulation point by the compactness of $\calM$, and this
accumulation point is at distance at least $\epsilon$ from
$\bu_{\mathrm{GS}}$, a contradiction. Thus $\bu^k\to\bu_{\mathrm{GS}}$. Since $\bE_h$ is
continuous, we get
$\bE_h(\bu^k)\to\bE_h(\bu_{\mathrm{GS}})=\min_{\calM}\bE_h$, the
equality again by Remark~\ref{d2:rem:gs}.
\end{proof}

Once the iterates are close to the ground state, the descent
inequality \eqref{d2:drop} implies a linear rate. The rate is a
statement for the fixed mesh, and neither the index from which it
holds nor its constant is uniform in $h$.

\begin{theorem}[Local linear convergence]\label{thm:rate}
Under the hypotheses of Theorem~\ref{thm:main} and with $N\ge2$, write
$E_{\mathrm{GS}}:=\bE_h(\bu_{\mathrm{GS}})$,
$\lambda_{\mathrm{GS}}:=\lambda(\bu_{\mathrm{GS}})$ and
$m_{\mathrm{GS}}:=\min_i(\bu_{\mathrm{GS}})_i>0$, let $\bn'$ be as in
\eqref{n:remainder} and $A_{\bu}$ as in \eqref{eq:A_u}, let
\begin{equation}
\calH:=-\Delta_h+\bn'(\bu_{\mathrm{GS}})-\lambda_{\mathrm{GS}}\bbI
=A_{\bu_{\mathrm{GS}}}-\lambda_{\mathrm{GS}}\bbI
+2\beta\diag(\bu_{\mathrm{GS}}^2),
\label{n:hessian}
\end{equation}
and let $\nu$ be the minimum of $\langle\bw,\calH\bw\rangle_h$
over all $\bw\in\calM$ with
$\langle\bu_{\mathrm{GS}},\bw\rangle_h=0$. Then
$\nu\ge2\beta m_{\mathrm{GS}}^2>0$, and with $\mu_{\max}$ the
largest eigenvalue of $-\Delta_h$ and
\begin{equation}
\eta:=\min\Bigl\{\frac{\gamma\nu}{4(1+\gamma\mu_{\max})},\
\frac12\Bigr\}\in(0,1),
\label{n:rate}
\end{equation}
there is an index $K$ such that, for every $k\ge K$,
\begin{equation}
\begin{gathered}
\calL^{k+1}-E_{\mathrm{GS}}\le(1-\eta)\bigl(\calL^k-E_{\mathrm{GS}}\bigr),
\\
\bE_h(\bu^k)-E_{\mathrm{GS}}+\frac{(R^k)^2}{\gamma}
+\frac{\nu}{4}\norm{\bu^k-\bu_{\mathrm{GS}}}_2^2
\le2\,(1-\eta)^{k-K}\bigl(\calL^K-E_{\mathrm{GS}}\bigr).
\end{gathered}
\label{n:geometric}
\end{equation}
\end{theorem}

\begin{proof}
We first prove the lower bound on $\nu$. By Appendix~\ref{app:fem},
$\bu_{\mathrm{GS}}>0$ is an eigenvector of the $h$-self-adjoint
operator $A_{\bu_{\mathrm{GS}}}$ belonging to its smallest
eigenvalue. Since $\nabla\bE_h(\bu)=A_{\bu}\bu$ by \S\ref{ssec:norms}
and $\bg(\bu_{\mathrm{GS}})=0$, \eqref{d2:grad} gives
$A_{\bu_{\mathrm{GS}}}\bu_{\mathrm{GS}}=\nabla\bE_h(\bu_{\mathrm{GS}})
=\lambda_{\mathrm{GS}}\bu_{\mathrm{GS}}$, thus this smallest
eigenvalue is $\lambda_{\mathrm{GS}}$, and since the Rayleigh
quotient of an $h$-self-adjoint operator is at least its smallest
eigenvalue, $\langle\bw,A_{\bu_{\mathrm{GS}}}\bw\rangle_h
\ge\lambda_{\mathrm{GS}}\norm{\bw}_2^2$ for every $\bw$. By
\eqref{n:hessian} and $(\bu_{\mathrm{GS}})_i^2\ge m_{\mathrm{GS}}^2$
for every $i$, we get
\begin{align*}
\langle\bw,\calH\bw\rangle_h
&=\langle\bw,A_{\bu_{\mathrm{GS}}}\bw\rangle_h-\lambda_{\mathrm{GS}}\norm{\bw}_2^2
+2\beta\langle\bw,\diag(\bu_{\mathrm{GS}}^2)\bw\rangle_h
\\
&\ge2\beta\langle\bw,\diag(\bu_{\mathrm{GS}}^2)\bw\rangle_h
\ge2\beta m_{\mathrm{GS}}^2\norm{\bw}_2^2,
\end{align*}
and taking the minimum over $\bw\in\calM$ with
$\langle\bu_{\mathrm{GS}},\bw\rangle_h=0$ gives
$\nu\ge2\beta m_{\mathrm{GS}}^2$.

Next we introduce a chart around the ground state. Let
$T_{\mathrm{GS}}:=\{\bxi:\langle\bu_{\mathrm{GS}},\bxi\rangle_h=0\}$ be the
tangent space of $\calM$ at $\bu_{\mathrm{GS}}$ and, for
$\bxi\in T_{\mathrm{GS}}$,
\[
\bv(\bxi):=\frac{\bu_{\mathrm{GS}}+\bxi}{\sqrt{1+\norm{\bxi}_2^2}}\in\calM,
\qquad
f(\bxi):=\bE_h(\bv(\bxi)) ,
\]
where $\norm{\bu_{\mathrm{GS}}+\bxi}_2^2=1+\norm{\bxi}_2^2$ by
$\norm{\bu_{\mathrm{GS}}}_2=1$ and
$\langle\bu_{\mathrm{GS}},\bxi\rangle_h=0$. Every $\bu\in\calM$ with
$\langle\bu,\bu_{\mathrm{GS}}\rangle_h>0$ is of this form: for
$\bxi=\bu/\langle\bu,\bu_{\mathrm{GS}}\rangle_h-\bu_{\mathrm{GS}}$ we
have $\langle\bu_{\mathrm{GS}},\bxi\rangle_h=1-1=0$,
$\norm{\bxi}_2^2=\langle\bu,\bu_{\mathrm{GS}}\rangle_h^{-2}-2+1
=\langle\bu,\bu_{\mathrm{GS}}\rangle_h^{-2}-1$ and
$\bv(\bxi)=\bigl(\bu/\langle\bu,\bu_{\mathrm{GS}}\rangle_h\bigr)
\big/\sqrt{\langle\bu,\bu_{\mathrm{GS}}\rangle_h^{-2}}=\bu$.
By $(1+p)^{-1/2}=1-\frac p2+O(p^2)$ with $p=\norm{\bxi}_2^2$,
\[
\bv(\bxi)-\bu_{\mathrm{GS}}=\bxi-\tfrac12\norm{\bxi}_2^2\bu_{\mathrm{GS}}
+O(\norm{\bxi}_2^3).
\]
The gradient of $\bE_h$ is $\nabla\bE_h(\bu)=-\Delta_h\bu+\bn(\bu)$
by \eqref{n:ndef}, and the Jacobian of $\bn$ is $\bn'$ of
\eqref{n:remainder}, thus the Hessian of $\bE_h$ with respect to
$\langle\cdot,\cdot\rangle_h$ is $-\Delta_h+\bn'(\bu)$. Since $\bE_h$
is a polynomial of degree four, its Taylor expansion at
$\bu_{\mathrm{GS}}$ reads, for every $\bv\in\bR^N$,
\begin{align*}
\bE_h(\bv)=E_{\mathrm{GS}}
&+\langle\nabla\bE_h(\bu_{\mathrm{GS}}),\bv-\bu_{\mathrm{GS}}\rangle_h
\\
&+\tfrac12\bigl\langle\bv-\bu_{\mathrm{GS}},\bigl(-\Delta_h+\bn'(\bu_{\mathrm{GS}})\bigr)(\bv-\bu_{\mathrm{GS}})\bigr\rangle_h
+O(\norm{\bv-\bu_{\mathrm{GS}}}_2^3).
\end{align*}
We take $\bv=\bv(\bxi)$, so that $\bE_h(\bv)=f(\bxi)$,
$\bv-\bu_{\mathrm{GS}}=\bxi-\tfrac12\norm{\bxi}_2^2\bu_{\mathrm{GS}}+O(\norm{\bxi}_2^3)$
and $\norm{\bv-\bu_{\mathrm{GS}}}_2=O(\norm{\bxi}_2)$. With
$\nabla\bE_h(\bu_{\mathrm{GS}})=\lambda_{\mathrm{GS}}\bu_{\mathrm{GS}}$,
$\langle\bu_{\mathrm{GS}},\bxi\rangle_h=0$ and
$\norm{\bu_{\mathrm{GS}}}_2=1$, the linear and the quadratic terms are
\begin{gather*}
\langle\nabla\bE_h(\bu_{\mathrm{GS}}),\bv-\bu_{\mathrm{GS}}\rangle_h
=\lambda_{\mathrm{GS}}\langle\bu_{\mathrm{GS}},\bv-\bu_{\mathrm{GS}}\rangle_h
=-\tfrac12\lambda_{\mathrm{GS}}\norm{\bxi}_2^2+O(\norm{\bxi}_2^3),
\\
\tfrac12\bigl\langle\bv-\bu_{\mathrm{GS}},\bigl(-\Delta_h+\bn'(\bu_{\mathrm{GS}})\bigr)(\bv-\bu_{\mathrm{GS}})\bigr\rangle_h
=\tfrac12\bigl\langle\bxi,\bigl(-\Delta_h+\bn'(\bu_{\mathrm{GS}})\bigr)\bxi\bigr\rangle_h
+O(\norm{\bxi}_2^3),
\end{gather*}
thus the Taylor expansion of $f$ at $\bxi=0$ is
\[
f(\bxi)=E_{\mathrm{GS}}-\tfrac12\lambda_{\mathrm{GS}}\norm{\bxi}_2^2
+\tfrac12\bigl\langle\bxi,\bigl(-\Delta_h+\bn'(\bu_{\mathrm{GS}})\bigr)\bxi\bigr\rangle_h
+O(\norm{\bxi}_2^3)
=E_{\mathrm{GS}}+\tfrac12\langle\bxi,\calH\bxi\rangle_h+O(\norm{\bxi}_2^3),
\]
where the last equality is \eqref{n:hessian}. Here $\bxi$ ranges
over the linear space $T_{\mathrm{GS}}$ and need not lie on $\calM$,
only the point $\bv(\bxi)$ does. Since $\bE_h$ is a polynomial and
$\bv$ is smooth on $T_{\mathrm{GS}}$, $f$ is a smooth function on
$T_{\mathrm{GS}}$. Since $f$ is defined on the inner product space
$T_{\mathrm{GS}}$, its gradient at $\bxi$ is the unique vector
$\nabla f(\bxi)\in T_{\mathrm{GS}}$ with
$\langle\nabla f(\bxi),\bw\rangle_h=\frac{d}{dt}f(\bxi+t\bw)\big|_{t=0}$
for all $\bw\in T_{\mathrm{GS}}$, and its Hessian at $\bxi$ is the
$h$-self-adjoint operator $\nabla^2f(\bxi)$ on $T_{\mathrm{GS}}$ with
$\langle\bw,\nabla^2f(\bxi)\bw\rangle_h=\frac{d^2}{dt^2}f(\bxi+t\bw)\big|_{t=0}$
for all $\bw\in T_{\mathrm{GS}}$. By the uniqueness of the second-order
Taylor polynomial of a smooth function, the expansion of $f$ at
$\bxi=0$ gives $\nabla f(0)=0$ and
$\langle\bw,\nabla^2f(0)\bw\rangle_h=\langle\bw,\calH\bw\rangle_h$ for
$\bw\in T_{\mathrm{GS}}$, which is at least $\nu\norm{\bw}_2^2$ by the
definition of $\nu$ applied to $\bw/\norm{\bw}_2$ for $\bw\neq0$.
Let $\mathbf e_1,\dots,\mathbf e_{N-1}$ be an
$\langle\cdot,\cdot\rangle_h$-orthonormal basis of $T_{\mathrm{GS}}$.
Expanding $\bw=\sum_{i=1}^{N-1}\langle\bw,\mathbf e_i\rangle_h\mathbf e_i$
in both arguments and using the Cauchy--Schwarz inequality for the
sums over the index pairs $(i,j)$, we have, for all
$\bxi,\bw\in T_{\mathrm{GS}}$,
\begin{align}
\bigl|\bigl\langle\bw,\bigl(\nabla^2f(\bxi)-\nabla^2f(0)\bigr)\bw\bigr\rangle_h\bigr|
&=\Bigl|\sum_{i,j=1}^{N-1}\langle\bw,\mathbf e_i\rangle_h\langle\bw,\mathbf e_j\rangle_h
\bigl\langle\mathbf e_i,\bigl(\nabla^2f(\bxi)-\nabla^2f(0)\bigr)\mathbf e_j\bigr\rangle_h\Bigr|
\notag\\
&\le\Bigl(\sum_{i,j=1}^{N-1}
\bigl\langle\mathbf e_i,\bigl(\nabla^2f(\bxi)-\nabla^2f(0)\bigr)\mathbf e_j\bigr\rangle_h^2
\Bigr)^{1/2}
\Bigl(\sum_{i,j=1}^{N-1}\langle\bw,\mathbf e_i\rangle_h^2\langle\bw,\mathbf e_j\rangle_h^2\Bigr)^{1/2},
\label{n:frobenius}
\end{align}
where the last factor equals $\norm{\bw}_2^2$, since
$\sum_{i,j=1}^{N-1}\langle\bw,\mathbf e_i\rangle_h^2\langle\bw,\mathbf e_j\rangle_h^2
=\bigl(\sum_{i=1}^{N-1}\langle\bw,\mathbf e_i\rangle_h^2\bigr)^2=\norm{\bw}_2^4$,
and the other factor is the Frobenius norm of the matrix of
$\nabla^2f(\bxi)-\nabla^2f(0)$ in this basis. Its entries are
continuous functions of $\bxi$ that vanish at $\bxi=0$, thus there is
$r>0$ such that the right side of \eqref{n:frobenius} is at most
$\frac\nu2\norm{\bw}_2^2$ for $\norm{\bxi}_2\le r$. Then, for all
$\bxi,\bw\in T_{\mathrm{GS}}$ with $\norm{\bxi}_2\le r$,
\[
\langle\bw,\nabla^2f(\bxi)\bw\rangle_h
=\langle\bw,\calH\bw\rangle_h
+\bigl\langle\bw,\bigl(\nabla^2f(\bxi)-\nabla^2f(0)\bigr)\bw\bigr\rangle_h
\ge\nu\norm{\bw}_2^2-\tfrac\nu2\norm{\bw}_2^2
=\tfrac\nu2\norm{\bw}_2^2 .
\]
For such $\bxi$, Taylor's formula with integral remainder along the
segment from $0$ to $\bxi$ and $\nabla f(0)=0$ give
\[
f(\bxi)-E_{\mathrm{GS}}
=\int_0^1(1-t)\langle\bxi,\nabla^2f(t\bxi)\bxi\rangle_h\,dt
\ge\tfrac\nu2\norm{\bxi}_2^2\int_0^1(1-t)\,dt=\tfrac\nu4\norm{\bxi}_2^2,
\]
and along the segment from $\bxi$ to $0$, together with the
Cauchy--Schwarz inequality and $as-\frac\nu4s^2\le\frac{a^2}{\nu}$
for real $a$ and $s$,
\begin{align*}
E_{\mathrm{GS}}-f(\bxi)
&=-\langle\nabla f(\bxi),\bxi\rangle_h
+\int_0^1(1-t)\langle\bxi,\nabla^2f((1-t)\bxi)\bxi\rangle_h\,dt
\\
&\ge-\norm{\nabla f(\bxi)}_2\norm{\bxi}_2+\tfrac\nu4\norm{\bxi}_2^2
\ge-\frac1\nu\norm{\nabla f(\bxi)}_2^2 ,
\end{align*}
Thus, for $\norm{\bxi}_2\le r$,
\begin{equation}
f(\bxi)-E_{\mathrm{GS}}\ge\tfrac\nu4\norm{\bxi}_2^2,
\qquad
f(\bxi)-E_{\mathrm{GS}}\le\frac1\nu\norm{\nabla f(\bxi)}_2^2 .
\label{n:strongconvex}
\end{equation}

Next we estimate the gradient and the distance in the chart. For
$\bw\in T_{\mathrm{GS}}$, differentiating $\bv(\bxi)$ in the direction
$\bw$ gives
\[
\bv'(\bxi)\bw=\bigl(1+\norm{\bxi}_2^2\bigr)^{-1/2}\bw
-\bigl(1+\norm{\bxi}_2^2\bigr)^{-3/2}
\langle\bxi,\bw\rangle_h(\bu_{\mathrm{GS}}+\bxi).
\]
Differentiating $\norm{\bv(\bxi)}_2^2=1$ gives
$\langle\bv(\bxi),\bv'(\bxi)\bw\rangle_h=0$. Expanding the square,
with $\langle\bu_{\mathrm{GS}}+\bxi,\bw\rangle_h=\langle\bxi,\bw\rangle_h$
and $\norm{\bu_{\mathrm{GS}}+\bxi}_2^2=1+\norm{\bxi}_2^2$, we have
\[
\norm{\bv'(\bxi)\bw}_2^2
=\bigl(1+\norm{\bxi}_2^2\bigr)^{-1}\norm{\bw}_2^2
-2\bigl(1+\norm{\bxi}_2^2\bigr)^{-2}\langle\bxi,\bw\rangle_h^2
+\bigl(1+\norm{\bxi}_2^2\bigr)^{-2}\langle\bxi,\bw\rangle_h^2
\le\norm{\bw}_2^2 .
\]
By the chain rule, $\nabla\bE_h(\bv)=\bg(\bv)+\lambda(\bv)\bv$ from
\eqref{d2:grad}, $\langle\bv(\bxi),\bv'(\bxi)\bw\rangle_h=0$, the
Cauchy--Schwarz inequality and $\norm{\bv'(\bxi)\bw}_2\le\norm{\bw}_2$,
\[
\langle\nabla f(\bxi),\bw\rangle_h
=\langle\nabla\bE_h(\bv(\bxi)),\bv'(\bxi)\bw\rangle_h
=\langle\bg(\bv(\bxi)),\bv'(\bxi)\bw\rangle_h
\le\norm{\bg(\bv(\bxi))}_2\norm{\bw}_2 ,
\]
and the choice $\bw=\nabla f(\bxi)\in T_{\mathrm{GS}}$ gives
$\norm{\nabla f(\bxi)}_2\le\norm{\bg(\bv(\bxi))}_2$. In
addition, by $\norm{\bv(\bxi)}_2=\norm{\bu_{\mathrm{GS}}}_2=1$,
$\langle\bxi,\bu_{\mathrm{GS}}\rangle_h=0$ and
$(1+p)^{-1/2}\ge1-\frac p2$ for $p\ge0$,
\[
\norm{\bv(\bxi)-\bu_{\mathrm{GS}}}_2^2
=2-2\langle\bv(\bxi),\bu_{\mathrm{GS}}\rangle_h
=2-2\bigl(1+\norm{\bxi}_2^2\bigr)^{-1/2}\le\norm{\bxi}_2^2 .
\]

Finally we prove the rate. By Theorem~\ref{thm:main},
$\bu^k\to\bu_{\mathrm{GS}}$, thus
$\langle\bu^k,\bu_{\mathrm{GS}}\rangle_h\to1$, and there is $K$ such
that, for all $k\ge K$, $\langle\bu^k,\bu_{\mathrm{GS}}\rangle_h>0$
and $\langle\bu^k,\bu_{\mathrm{GS}}\rangle_h^{-2}-1\le r^2$, i.e.,
$\bu^k=\bv(\bxi_k)$ with $\bxi_k\in T_{\mathrm{GS}}$ and
$\norm{\bxi_k}_2\le r$ by the chart. For such $k$,
\eqref{n:strongconvex},
$\norm{\nabla f(\bxi_k)}_2\le\norm{\bg(\bu^k)}_2$, \eqref{n:Xequiv}
and $\norm{\bu^k-\bu_{\mathrm{GS}}}_2\le\norm{\bxi_k}_2$ give
\begin{gather*}
\bE_h(\bu^k)-E_{\mathrm{GS}}\le\frac1\nu\norm{\bg(\bu^k)}_2^2
\le\frac{1+\gamma\mu_{\max}}{\nu}\norm{\bg(\bu^k)}_X^2,
\\
\norm{\bu^k-\bu_{\mathrm{GS}}}_2^2\le\norm{\bxi_k}_2^2
\le\frac4\nu\bigl(\bE_h(\bu^k)-E_{\mathrm{GS}}\bigr).
\end{gather*}
Plugging the first bound into the first inequality of
\eqref{d2:drop}, and using
$\eta\le\frac{\gamma\nu}{4(1+\gamma\mu_{\max})}$ and $\eta\le\frac12$
from \eqref{n:rate}, $\bE_h(\bu^k)-E_{\mathrm{GS}}\ge0$ and
$\calL^k-E_{\mathrm{GS}}=\bE_h(\bu^k)-E_{\mathrm{GS}}+(R^k)^2/\gamma$
from \eqref{d2:orbit}, we obtain
\[
\calL^k-\calL^{k+1}
\ge\frac{\gamma\nu}{4(1+\gamma\mu_{\max})}
\bigl(\bE_h(\bu^k)-E_{\mathrm{GS}}\bigr)
+\frac12\,\frac{(R^k)^2}{\gamma}
\ge\eta\bigl(\calL^k-E_{\mathrm{GS}}\bigr),
\]
and subtracting both sides from $\calL^k-E_{\mathrm{GS}}$ gives the
first inequality in \eqref{n:geometric}. By induction,
$\calL^k-E_{\mathrm{GS}}\le(1-\eta)^{k-K}(\calL^K-E_{\mathrm{GS}})$ for
$k\ge K$, and adding
$\frac\nu4\norm{\bu^k-\bu_{\mathrm{GS}}}_2^2
\le\bE_h(\bu^k)-E_{\mathrm{GS}}\le\calL^k-E_{\mathrm{GS}}$ to
$\calL^k-E_{\mathrm{GS}}=\bE_h(\bu^k)-E_{\mathrm{GS}}+(R^k)^2/\gamma$
gives the second inequality in \eqref{n:geometric}.
\end{proof}

\section{Numerical Experiments}\label{sec:numerics}

In this section, we test two splitting schemes \eqref{dys} and \eqref{dys2} numerically, and compare them with five other efficient 
iterative methods: the $H^1$ gradient
flow and the Riemannian conjugate gradient method in the same $H^1$ metric,  the preconditioned
gradient (PG) and preconditioned conjugate gradient (PCG) methods in \cite{antoine2017efficient}, and the energy-adaptive $a_u$-flow \cite{henning2020sobolev}. Two spatial discretizations of
$\Omega=[-8,8]^3$ will be used. The first one is the second-order finite difference
scheme, satisfying Assumption~\ref{asp:struct}, on a uniform grid with $n^3$ interior
points, so that $\bbM=h^3\bbI$ with $h=16/(n+1)$, at $n=799$. The second discretization is the
$Q^k$ spectral element method with the diagonal
mass matrix of Gauss--Lobatto quadrature \cite{liu2023simple, chen2024fully}. 
In numerical tests below, we use the
$Q^{20}$ spectral element method on $50$ cells per direction, so that the
3D problem has $999^3$ unknowns. The $Q^{20}$ scheme is not monotone, so
these runs lie outside the hypotheses of Theorem~\ref{thm:main}. By considering such a discretization, we would like to show that the performance belongs to the splitting schemes and not uniquely to one
discretization.

In every experiment $\beta=1600$, and every method built from
$\bbG_\alpha:=(\alpha\bbI-\Delta_h)^{-1}$ uses the shift $\alpha=20$. All
methods solve their linear systems by the same simple
solver in \cite{liu2023simple,liu2026gpu}. All tests were run in Python with
JAX on one NVIDIA GH200 GPU card, in double precision.

\subsection{The methods}

\emph{$H^1$ flow with line search} is the discrete $H^1$ gradient flow
\cite{kazemi2010minimizing,danaila2010new,chen2024fully}, that is, Riemannian
gradient descent on $\calM$ for the shifted $H^1$ metric
$\langle\bv,\bw\rangle_\alpha:=\langle\bv,(\alpha\bbI-\Delta_h)\bw\rangle_h$.
The Riemannian gradient of $\bE_h$ at $\bu\in\calM$ for this metric is
\begin{equation}
\bg_\alpha(\bu)
:=\bbG_\alpha\nabla\bE_h(\bu)
-\frac{\langle\bbG_\alpha\nabla\bE_h(\bu),\bu\rangle_h}
{\langle\bbG_\alpha\bu,\bu\rangle_h}\,\bbG_\alpha\bu ,
\label{num:rgrad}
\end{equation}
the $\langle\cdot,\cdot\rangle_\alpha$-orthogonal projection of
$\bbG_\alpha\nabla\bE_h(\bu)$ onto the tangent plane of $\calM$ at $\bu$, and
one step is
\begin{equation}
\bu^{k+1}
=\frac{\bu^k-\tau_k\,\bg_\alpha(\bu^k)}
{\norm{\bu^k-\tau_k\,\bg_\alpha(\bu^k)}_2} ,
\label{num:rgf}
\end{equation}
which applies $\bbG_\alpha$ twice, to $\nabla\bE_h(\bu^k)$ and to $\bu^k$. The
step size $\tau_k$ is chosen at every iteration by minimizing the energy of the
normalized update.

\emph{Riemannian CG} is the Riemannian conjugate gradient method of Danaila and
Protas \cite{danaila2017computation} for the same metric: the search direction
at $\bu^k$ is $-\bg_\alpha(\bu^k)$ plus the previous direction, transported to
the tangent plane at $\bu^k$ by the projection in \eqref{num:rgrad} and
multiplied by the Fletcher--Reeves ratio
$\langle\bg_\alpha(\bu^k),\bg_\alpha(\bu^k)\rangle_\alpha/
\langle\bg_\alpha(\bu^{k-1}),\bg_\alpha(\bu^{k-1})\rangle_\alpha$, and the step
size comes from an Armijo search started at the one-dimensional minimizer.

\emph{Shifted Laplacian preconditioned GF} and \emph{CG} are the PG and PCG
methods of \cite{antoine2017efficient} with $\bbG_\alpha$ as the
preconditioner: the same shifted Laplacian that defines the metric above, used
as a preconditioner rather than as a metric, so the two families differ in how
they use the operator and not in which one they invert. Both take the projected
gradient $\bg(\bu^k)$
of \eqref{d2:grad} as the residual, with directions
$\bd^k=-\bbG_\alpha\,\bg(\bu^k)$ for PG and
\[
\bd^k=-\bbG_\alpha\,\bg(\bu^k)
+\frac{\langle\bg(\bu^k),\bbG_\alpha\bg(\bu^k)\rangle_h}
{\langle\bg(\bu^{k-1}),\bbG_\alpha\bg(\bu^{k-1})\rangle_h}\,
\bp^{k-1}
\]
for PCG. Both project the direction onto the tangent plane,
$\bp^k=\calP_{\bu^k}\bd^k$, and update along the great circle through $\bu^k$,
\begin{equation}
\bu^{k+1}=\cos(\theta_k)\,\bu^k
+\sin(\theta_k)\,\frac{\bp^k}{\norm{\bp^k}_2} ,
\label{num:pg}
\end{equation}
with $\theta_k$ minimizing $\bE_h(\bu^{k+1})$ over $[-\pi,\pi]$. Each iteration
applies $\bbG_\alpha$ once.

\emph{$a_u$-flow} is the energy-adaptive Riemannian gradient flow of
\cite{henning2020sobolev}, for the metric induced by $A_\bu$ itself,
$\langle\bv,\bw\rangle_{a_u}=\langle\bv,A_{\bu}\bw\rangle_h$.
Because $\nabla\bE_h(\bu)=A_{\bu}\bu$, applying
$A_{\bu}^{-1}$ to the gradient costs nothing, so one solve per
iteration suffices rather than two, as recorded in
Table~\ref{tab:cost}: with $\bw^k$ the solution of
$A_{\bu^k}\bw^k=\bu^k$ and
$\rho_k=\langle\bu^k,\bu^k\rangle_h/\langle\bu^k,\bw^k\rangle_h$,
the step is
\begin{equation}
\bu^{k+1}=\frac{\bu^k+\tau_k(\rho_k\bw^k-\bu^k)}
{\norm{\bu^k+\tau_k(\rho_k\bw^k-\bu^k)}_2},
\label{num:au}
\end{equation}
which at $\tau_k=1$ is a normalized inverse iteration in the $a_u$ metric.
The matrix $A_{\bu^k}$ moves with the iterate and is not separable, so this is
the one linear system in the comparison that cannot be diagonalized. It is
solved by conjugate gradients preconditioned by $(-\Delta_h+\bbV_1)^{-1}$ as in
\cite{liu2026gpu}, stopped at relative residual $10^{-8}$. The step size is
taken from \cite[Remark 4.3]{henning2020sobolev} rather than fixed at $1$:
along the update ray the energy of the normalized iterate is a ratio of
polynomials in $\tau$ whose eleven coefficients are inner products of $\bu^k$
and $\rho_k\bw^k$, gathered once per step, after which every trial $\tau$ is
a scalar evaluation and the minimization over $(0,2)$ needs no further solve.
Without this rule the $a_u$-flow would be the only method here carrying an
untuned step size: at $n=199$ it converges in $54$ iterations against $77$ at
$\tau_k=1$, and against $69$ at the best constant $\tau$ found by a sweep.

\emph{Davis--Yin I} is the scheme \eqref{dys}, whose single solve per iteration
inverts $\bbI-\gamma_k\Delta_h$, with
\begin{equation}
L_k:=\max_i\bigl(V_i+3\beta(u_i^k)^2\bigr),
\label{num:Lk}
\end{equation}
the Lipschitz constant at the current iterate of
$\nabla H$, where
$H(\bu)=\tfrac12\langle\bbV\bu,\bu\rangle_h+\tfrac\beta4\norm{\bu}_4^4$
is the explicitly treated part of the splitting in
Appendix~\ref{app:dys}, since
$\nabla^2H(\bu)=\mathrm{diag}\bigl(V_i+3\beta u_i^2\bigr)$. The choice
$\gamma_k=1/L_k$ is already larger than the threshold $\gamma_{\max}$
of Theorem~\ref{thm:main}, which is what the proof requires rather
than what the method tolerates. Even $1/L_k$ is conservative by about
a factor of two, so in practice we use $\gamma_k=c/L_k$. Different
constants give different
performance, and the values used here were found by tuning on coarse grids:
for the three examples of \S\ref{ssec:tests}, $c=1.6$ on the lattice from the
constant initial guess, $1.7$ on the lattice from the linear one and $2.1$ on
the stirrer. \emph{Davis--Yin II} is the scheme \eqref{dys2}, the same
iteration with $-\Delta_h+\bbV_1$ inside the resolvent and with
$L_k=\max_i(V_{2,i}+3\beta(u^k_i)^2)$. Its constants, tuned the same
way, are smaller, $1.4$ on the lattice from either initial guess and $1.2$ on
the stirrer, since a smaller $L_k$ comes with a tighter limit on $c$ and only
the product $\gamma_k=c/L_k$ matters.

Five of the seven methods choose their step size by a
one-dimensional minimization and the two Davis--Yin schemes do not, as
Table~\ref{tab:linesearch} records. Every such search is evaluated in closed
form, the energy along the search direction being a ratio of polynomials in the
step, so that no line search costs a linear solve. Table~\ref{tab:mem} records
what one iteration of each method costs, in elliptic solves and in memory.

\begin{table}[htbp]
\centering
\footnotesize
\setlength{\tabcolsep}{4pt}
\begin{tabular}{lll}
\hline
method & step size & one-dimensional minimization\\
\hline
$H^1$ flow with line search & $\tau_k$ minimizing $\bE_h$ in \eqref{num:rgf}
  & along the ray, over $\tau\in(0,5]$\\
Riemannian CG & Armijo from that minimizer
  & along the ray, then a backtrack\\
Shifted Laplacian preconditioned GF & $\theta_k$ minimizing $\bE_h$ in \eqref{num:pg}
  & along the great circle, $[-\pi,\pi]$\\
Shifted Laplacian preconditioned CG & as above
  & along the great circle\\
DYS I & $\gamma_k=c/L_k$, $L_k=\max_i(V_i+3\beta(u^k_i)^2)$
  & none\\
DYS II & $\gamma_k=c/L_k$, $L_k=\max_i(V_{2,i}+3\beta(u^k_i)^2)$
  & none\\
$a_u$-flow & $\tau_k$ minimizing $\bE_h$ in \eqref{num:au}
  & along the update ray, over $(0,2)$\\
\hline
\end{tabular}
\caption{The step size rules used in numerical tests.}
\label{tab:linesearch}
\end{table}

\begin{table}[htbp]
\centering
\footnotesize
\setlength{\tabcolsep}{4pt}
\begin{tabular}{lccc}
\hline
 & vectors carried & elliptic solves & memory cost per iteration of the\\
 & across an & per iteration, & simple efficient implementation\\
method & iteration & and of what & used in this paper, on a $999^3$ grid\\
\hline
DYS I \eqref{dys}
 & $2N$ & $1\times(\bbI-\gamma_k\Delta_h)^{-1}$ & $8\,n^3+2n^2$: \ \ $63.8$ GB\\
DYS II \eqref{dys2}
 & $2N$ & $1\times(\bbI-\gamma_k\Delta_h+\gamma_k\bbV_1)^{-1}$ & $8\,n^3+2n^2$: \ \ $63.8$ GB\\
Shifted Laplacian precond.\ CG
 & $2N$ & $1\times(\alpha\bbI-\Delta_h)^{-1}$ & $7\,n^3+2n^2$: \ \ $55.8$ GB\\
Shifted Laplacian precond.\ GF
 & $2N$ & $1\times(\alpha\bbI-\Delta_h)^{-1}$ & $9\,n^3+2n^2$: \ \ $71.8$ GB\\
$H^1$ flow with line search
 & $2N$ & $2\times(\alpha\bbI-\Delta_h)^{-1}$ & $9\,n^3+2n^2$: \ \ $71.8$ GB\\
Riemannian CG
 & $2N$ & $2\times(\alpha\bbI-\Delta_h)^{-1}$ & $8\,n^3+2n^2$: \ \ $63.8$ GB\\
$a_u$-flow
 & $1N$ & $1\times A_{\bu^k}^{-1}$ & $10\,n^3+2n^2$: \ $79.8$ GB\\
\hline
\end{tabular}
\caption{What one iteration costs, with $N=n^3$ degrees of freedom. The first
two columns belong to the schemes and hold for any implementation. The last is
what this one reaches, measured one method per process on one GH200 in double
precision.  The
$2n^2$ is the pair of $n\times n$ factors of the per-axis eigendecomposition, which is needed in the simple GPU implementation of Poisson solver \cite{liu2023simple, liu2026gpu}.}
\label{tab:mem}
\end{table}

\subsection{The two potentials and the three examples}\label{ssec:tests}

The figures below plot the energy error $|\bE_h(\bu^k)-E^*|$ against
the iteration count and against the wall-clock time. The reference energy $E^*$
is generated numerically on the same grid by running any one of the methods for
enough iterations, well past the accuracy that the compared runs reach, so that
no plotted curve is limited by it.

The first potential is an optical lattice superposed on a harmonic trap,
\begin{equation}
V_{\mathrm{lat}}(\bx)
=|\bx|^2+100\sum_{j=1}^{3}\sin^2\!\Bigl(\frac{\pi x_j}{4}\Bigr),
\label{num:Vlat}
\end{equation}
which is additively separable. The second is an anisotropic harmonic trap
stirred by a Gaussian beam directed along $x_3$,
\begin{equation}
V_{\mathrm{stir}}(\bx)
=x_1^2+x_2^2+\omega_z^2x_3^2
+2w_0\,e^{-\delta\left((x_1-r_0)^2+x_2^2\right)},
\qquad w_0=30,\ \delta=1,\ r_0=1,\ \omega_z=2,
\label{num:Vstir}
\end{equation}
a standard model of a laser stirrer in a condensate
\cite{bao2004computing}. The beam couples $x_1$ and
$x_2$, so $V_{\mathrm{stir}}$ is not separable. Writing $V=V_1+V_2$ with $V_1$
the separable part, the Gaussian $V_2$ is a rank-one outer product in
$(x_1,x_2)$. The scheme \eqref{dys} treats $V$ explicitly and is indifferent to
this distinction. The scheme \eqref{dys2} and the $a_u$-flow are not, the
first inverting $-\Delta_h+\bbV_1$ and the second using it as a
preconditioner. Both need $V_1$ to be separable, and
$-\Delta_h+\bbV_1$ is then inverted by the fast diagonalization of
\cite{liu2026gpu}.
Figure~\ref{fig:states} shows both traps and the ground states they support.

\begin{figure}[htbp]
    \centering
    \includegraphics[width=0.85\linewidth]{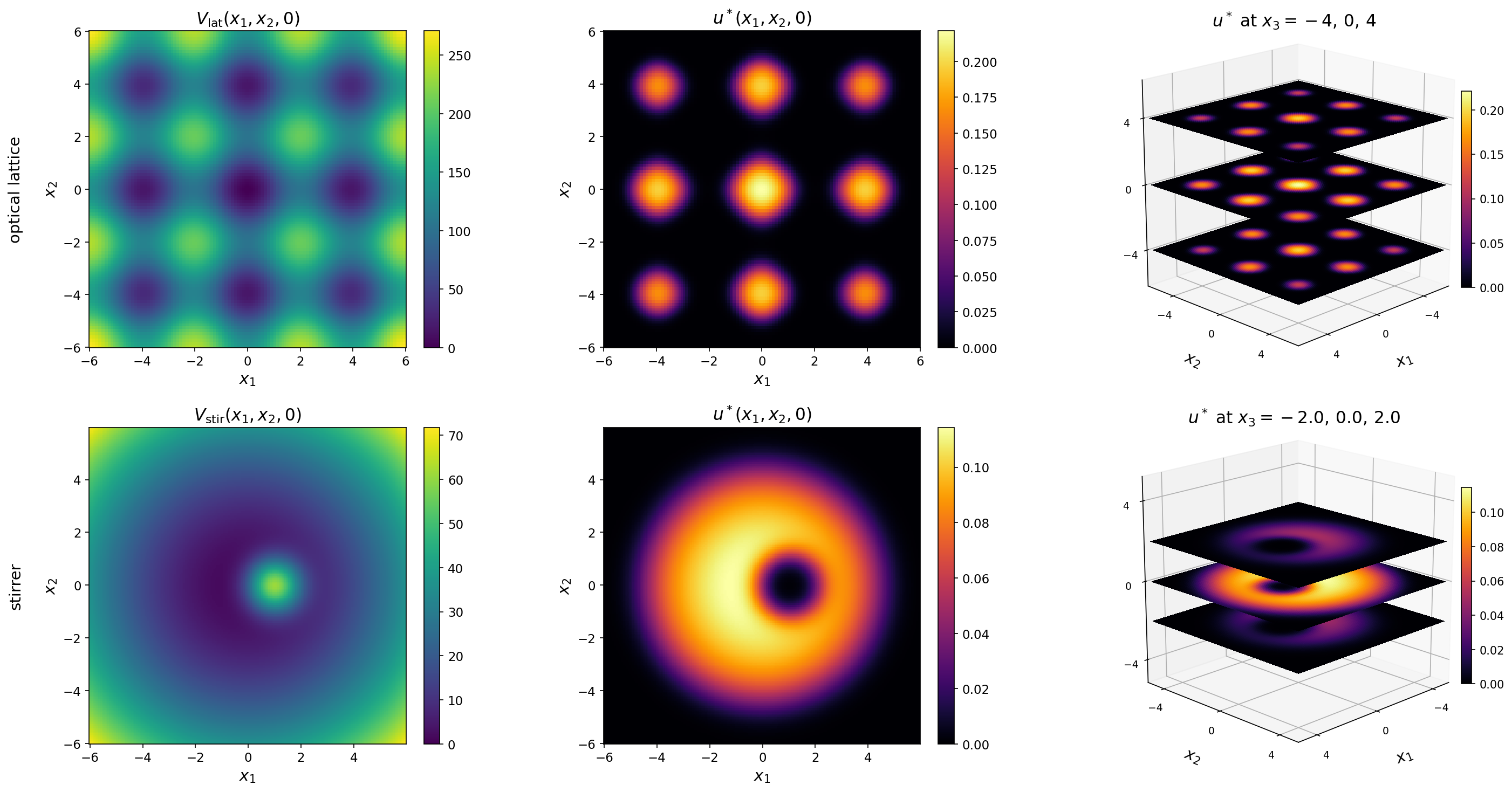}
    \caption{The two potentials and their ground states: the optical lattice
    \eqref{num:Vlat} on top, the stirrer \eqref{num:Vstir} below. Left is the
    trap on the plane $x_3=0$, middle the ground state on the same plane, right
    the ground state on three planes.}
    \label{fig:states}
\end{figure}

\begin{example}[Optical lattice, constant initial guess]\label{ex:lattice}
The potential is \eqref{num:Vlat} and every method starts from the normalized
constant vector, which is positive as Theorem~\ref{thm:main} requires.
Figure~\ref{fig:ex1} reports the comparison, at $n=799$ with finite differences
and at $n=999$ with spectral elements. The practical computational
efficiency of the two Davis--Yin splitting schemes is comparable to that of the
CG methods, even though they need more iterations.
\end{example}

\begin{figure}[htbp]
    \centering
    \includegraphics[width=0.85\linewidth]{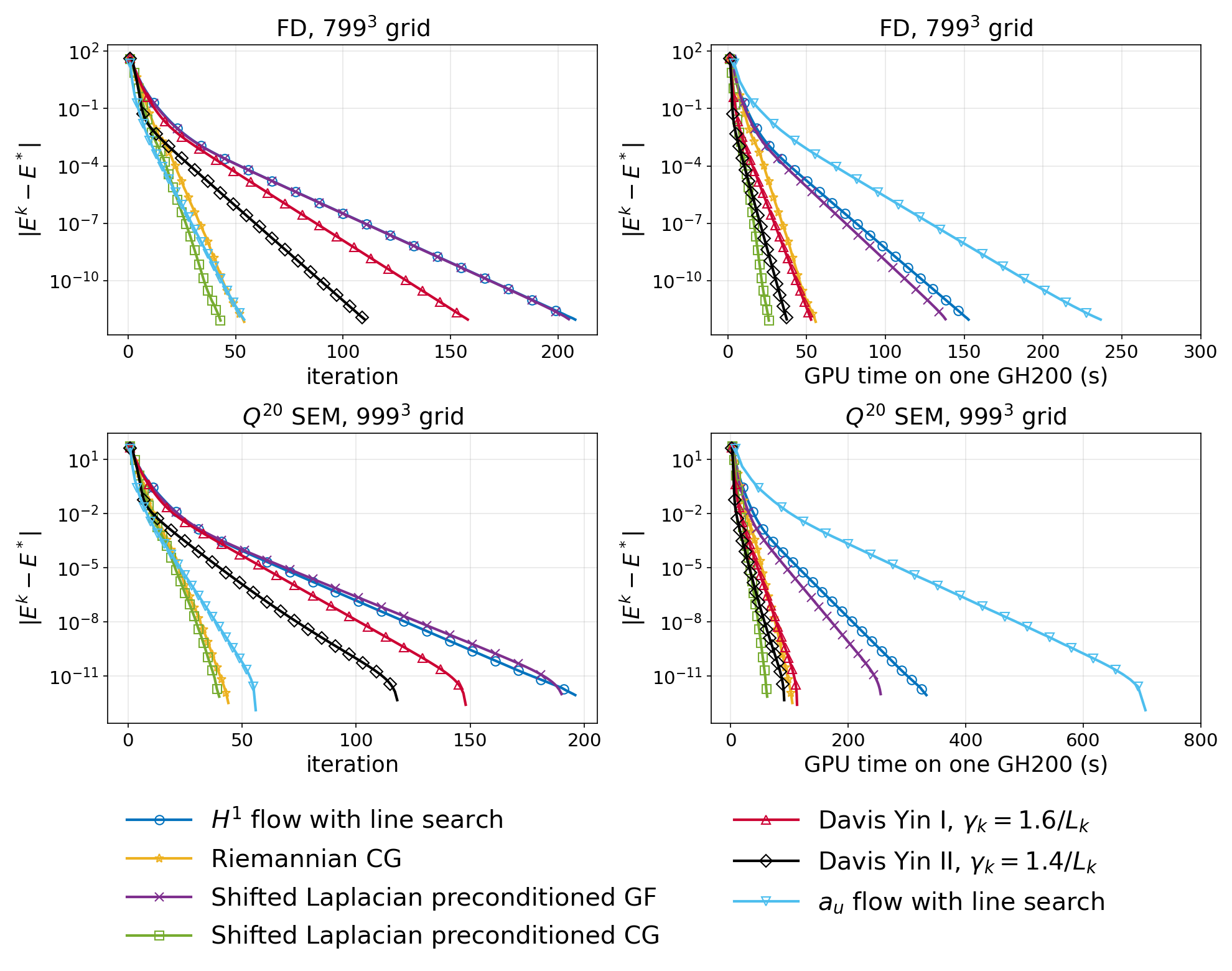}
    \caption{Example~\ref{ex:lattice}, the optical lattice from the constant
    initial guess: second-order finite differences (FD) on a $799^3$
    grid on top, $Q^{20}$ spectral elements (SEM) on a $999^3$
    grid below.}
    \label{fig:ex1}
\end{figure}

\begin{example}[Stirrer, constant initial guess]\label{ex:stirrer}
The potential is \eqref{num:Vstir}, which is not separable, and the initial
guess is again the normalized constant vector. Figure~\ref{fig:ex2} reports the
comparison on the same two grids. The stirring beam is explicit in
\eqref{dys}, and its peak sets $L_k$ and with it the step size, so here
\eqref{dys} is slower than the CG methods. The scheme \eqref{dys2} keeps only
the beam explicit, and is the fastest of the seven methods in wall-clock time.
\end{example}

\begin{figure}[htbp]
    \centering
    \includegraphics[width=0.85\linewidth]{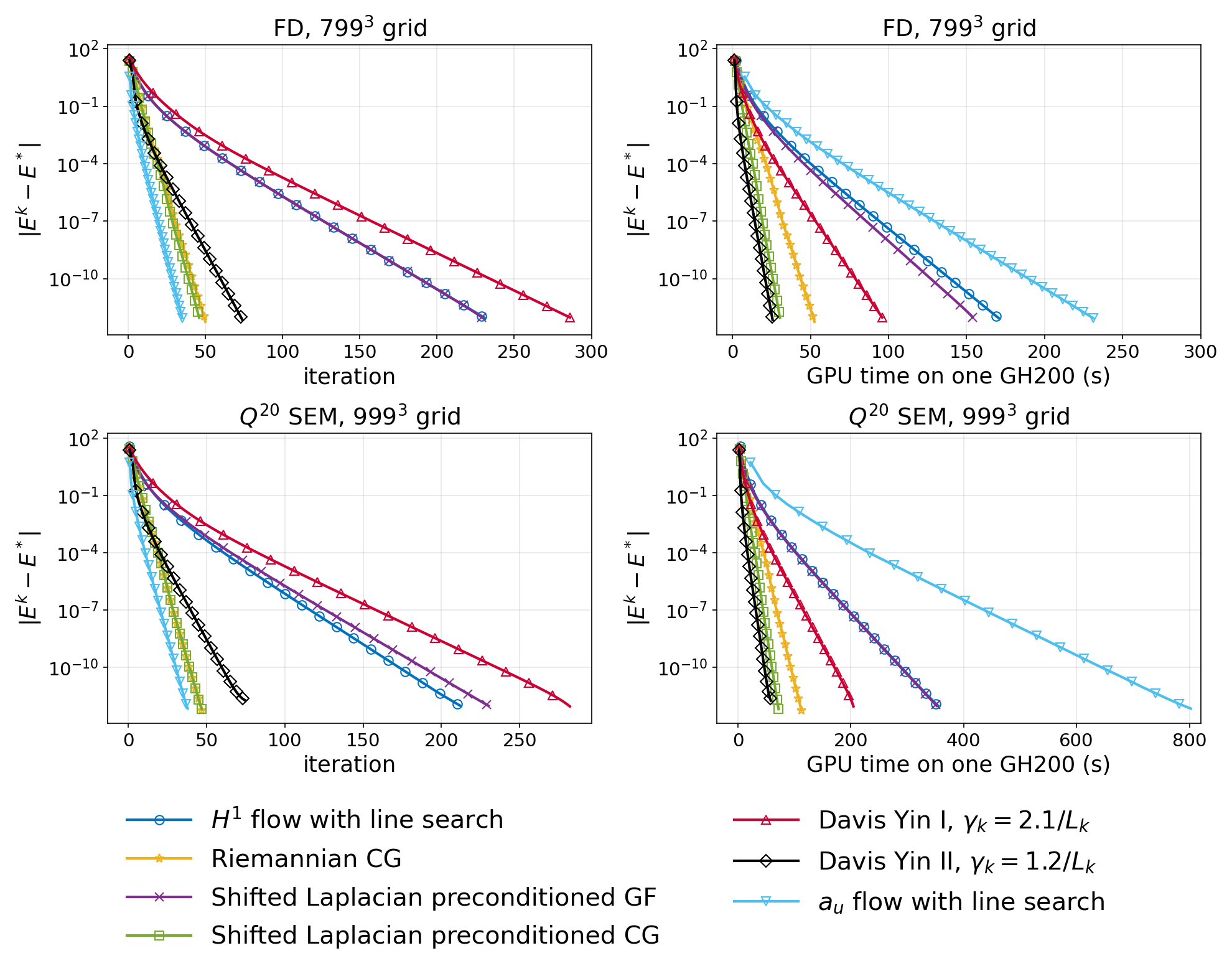}
    \caption{Example~\ref{ex:stirrer}, the stirrer potential, in the panels
    of Figure~\ref{fig:ex1}. DYS~II is the fastest of the seven
    methods here.}
    \label{fig:ex2}
\end{figure}

\begin{example}[Optical lattice, linear ground state initial
guess]\label{ex:linear}
The potential is \eqref{num:Vlat} and every method starts from the ground state
of the linear problem, the eigenvector of $-\Delta_h+\bbV$ belonging to its
smallest eigenvalue, computed by shifted inverse iteration. This is a natural
and widely used initialization for smaller $\beta$, but at $\beta=1600$ its energy exceeds $E^*$ by
a factor of about sixteen. Figure~\ref{fig:ex3} shows the initial guess and the
comparison. Unlike the conjugate gradient methods, which are sensitive to this poor initial guess and
 stall within the first few hundred iterations at states
that are not the ground state, the schemes \eqref{dys} and \eqref{dys2}
converge very efficiently.
\end{example}

\begin{figure}[htbp]
    \centering
    \subfigure[The linear ground state ($\beta=0$), and the ground state for
    $\beta=1600$ for the lattice potential.]{%
        \includegraphics[width=0.85\linewidth]{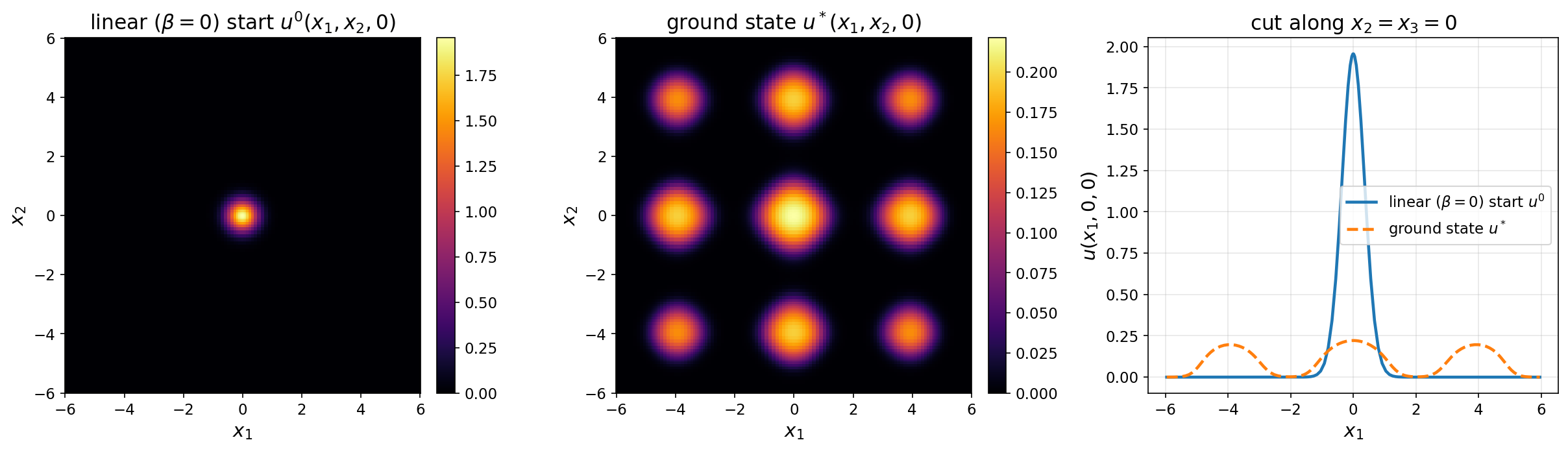}
        \label{fig:ex3init}}\\[6pt]
    \subfigure[The comparison for second order finite difference on a
    $799^3$ grid.]{%
        \includegraphics[width=0.85\linewidth]{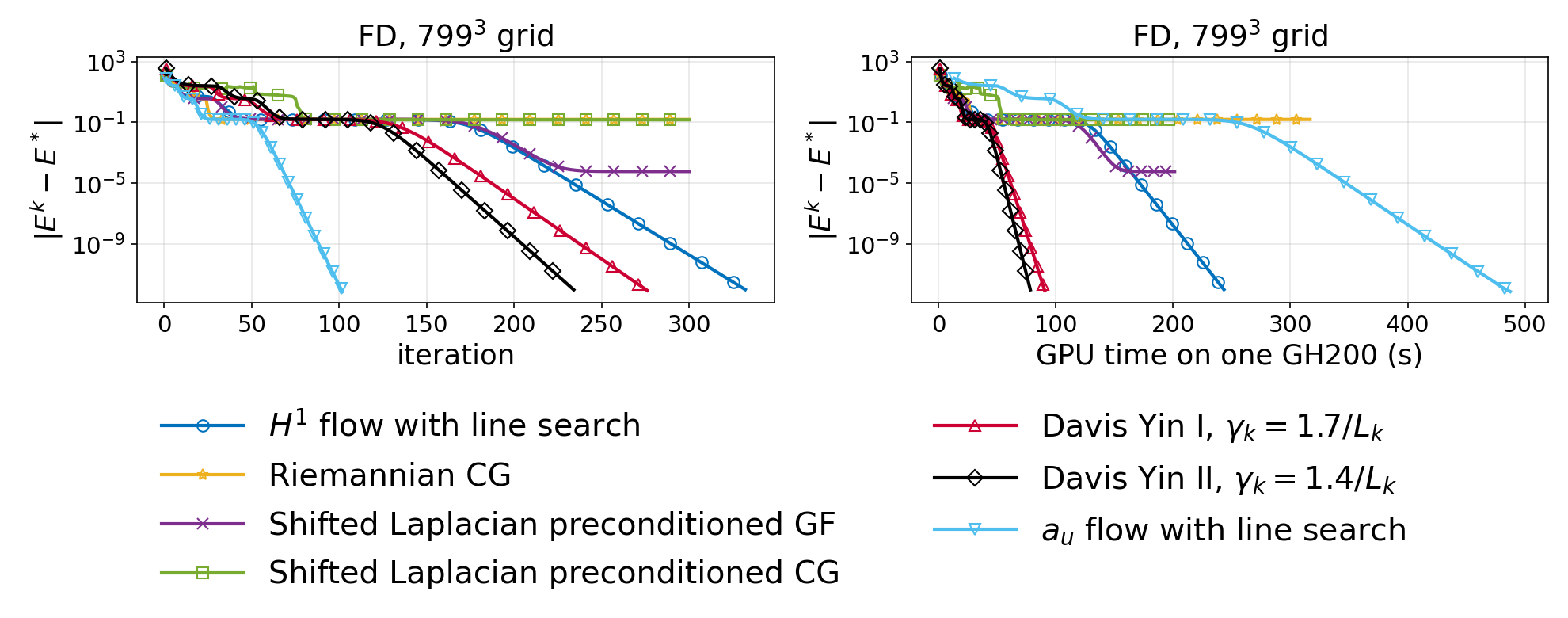}
        \label{fig:ex3n799}}
    \caption{At $\beta=1600$ the ground state of the linear problem is a poor
    initial guess, which is used in Example~\ref{ex:linear} only to test how
    much methods depend on a good initial guess. The CG methods
    perform poorly, while the DYS schemes are superior in computational
    time.}
    \label{fig:ex3}
\end{figure}

\section{Concluding remarks}\label{sec:conclusion}

We  have proved global convergence of the semi-implicit Davis--Yin
splitting scheme \eqref{dys} for the discrete defocusing
Gross--Pitaevskii ground state problem: for the monotone
discretizations of Section~\ref{sec:fem}, for every positive
normalized initial vector, and for every constant step size below
the explicit threshold of Theorem~\ref{thm:main}, the normalized
iterates converge to the unique positive discrete ground state,
with no spectral gap assumption and no requirement that the
initial guess be close to it. In the numerical tests of
Section~\ref{sec:numerics}, the simple step size rule
$\gamma_k=c/L_k$ of \eqref{num:Lk}, which costs no line search
and no inner iteration, makes the schemes \eqref{dys} and
\eqref{dys2} efficient in practice: on the examples where every
method converges, their wall-clock performance is comparable to
that of the conjugate gradient methods, and they are much more
robust with respect to the choice of the initial guess.  

\appendix

\section{The monotone spatial discretizations}
\label{app:fem}

This appendix collects the construction of the two discretizations of
Section~\ref{sec:fem} and the monotonicity properties they enjoy,
following \cite{chen2024fully}.

\subsection{Finite element method with quadrature}
Let $\Omega_h$ be a mesh of $\Omega$, which is either
\begin{enumerate}
\item a uniform rectangular mesh, on which we use the space of
continuous piecewise $Q^1$ (multilinear) polynomials, or
\item an unstructured simplicial mesh, e.g., a triangular mesh for
$d=2$, on which we use the space of continuous piecewise $P^1$
(linear) polynomials.
\end{enumerate}
In both cases the finite element space with the homogeneous
Dirichlet boundary condition is
\[
V_0^h=\{v_h\in C(\Omega):\ v_h|_{\partial\Omega}=0,\
v_h|_e\in Q^1(e)\ \text{(resp.\ }P^1(e)\text{)},\
\forall\,e\in\Omega_h\}\subset H_0^1(\Omega).
\]
In practice the integrals defining the scheme are evaluated by
quadrature. For the $Q^1$ method we use the $2$-point Gauss--Lobatto
(i.e., trapezoidal) rule in each dimension on each cell, and for the
$P^1$ method the vertex rule
$\int_e f\,\mathrm{d}\bx\approx\frac{|e|}{d+1}\sum_{\bx_v\in e}
f(\bx_v)$ on each simplex, i.e., mass lumping. (For $P^1$ elements
the stiffness integrand $\nabla u_h\cdot\nabla v_h$ is piecewise
constant, so it is integrated exactly, and only the mass and
potential terms are lumped.) Let $\langle\cdot,\cdot\rangle$ denote the
resulting quadrature approximation of $(\cdot,\cdot)$. The finite
element method with quadrature for \eqref{variational} is to seek
$\lambda_h\in\bR$ and $u_h\in V_0^h$ satisfying
\begin{equation}
\langle\nabla u_h,\nabla v_h\rangle+\langle Vu_h,v_h\rangle
+\beta\langle|u_h|^2u_h,v_h\rangle
=\lambda_h\langle u_h,v_h\rangle,
\qquad\forall\,v_h\in V_0^h,
\label{fem2}
\end{equation}
with the corresponding discrete energy
\begin{equation}
E_h(u_h)=\frac12\langle\nabla u_h,\nabla u_h\rangle
+\frac12\langle Vu_h,u_h\rangle
+\frac\beta4\langle u_h^2,u_h^2\rangle,
\label{fem-energy}
\end{equation}
minimized under the normalization constraint
$\langle u_h,u_h\rangle=1$. On a uniform rectangular mesh the $Q^1$
scheme with quadrature coincides with the classical second-order
central finite difference scheme, see
\cite[Appendix]{chen2024fully} for the explicit expressions of both
schemes. Using such quadrature does not degrade the standard
\emph{a priori} error estimates of the finite element method
\cite{ciarletbook}. For the convergence of finite element
approximations of the nonlinear eigenvalue problem
\eqref{continuum}, see \cite{cances2010numerical,chen2024fully}.

\subsection{Matrix--vector form}
With the notation of Section~\ref{sec:fem}, write $u_i=u_h(\bx_i)$
and $u_h=\sum_{i=1}^Nu_i\phi_i$ for $u_h\in V_0^h$. Then
\begin{equation}
\langle Vu_h,v_h\rangle=\sum_{i=1}^Nw_iV_iu_iv_i
=\bv^\top\bbM\bbV\bu,
\qquad
\langle\nabla u_h,\nabla v_h\rangle=\bv^\top\bbS\bu ,
\label{part12}
\end{equation}
so that \eqref{fem2} is exactly \eqref{fd3} and \eqref{fem-energy} is
exactly \eqref{fd-energy}. The stiffness matrix $\bbS$ is symmetric
positive definite, and $\bbM$ is diagonal with positive entries.

\subsection{Monotonicity}
A matrix $A$ is called \emph{monotone} if $A^{-1}\ge0$ entrywise. A
symmetric positive definite matrix with nonpositive off-diagonal
entries is a (symmetric) M-matrix, also called a Stieltjes matrix,
and is monotone, see \cite{varga1999matrix,plemmons1977m}. Both
discretizations above are monotone in the following sense.
\begin{enumerate}
\item For the $Q^1$ scheme on a uniform rectangular mesh
(equivalently, the second-order finite difference scheme), the
stiffness matrix $\bbS$ has nonpositive off-diagonal entries, which
is classical.
\item For the $P^1$ scheme on a simplicial mesh
$\Omega_h\subset\bR^d$, let $E$ denote an interior edge connecting
the vertices $\bx_i$ and $\bx_j$, let $\kappa_E^T$ be the
$(d-2)$-dimensional simplex opposite to $E$ in a simplex
$T\supset E$, and let $\theta_E^T$ be the dihedral angle between
the two faces of $T$ containing $E$. Then
$\bbS_{ij}=-\sum_{T\supset E}\frac1{d(d-1)}|\kappa_E^T|
\cot\theta_E^T$, so the off-diagonal entries of $\bbS$ are
nonpositive if and only if the mesh satisfies
\begin{equation}
\sum_{T\supset E}\frac1{d(d-1)}|\kappa_E^T|\cot\theta_E^T\ \ge\ 0
\qquad\text{for every interior edge }E,
\label{simpicialmesh}
\end{equation}
see \cite[Lemma~2.1]{xu1999monotone}. For $d=2$ the condition
\eqref{simpicialmesh} reduces to
$\cot\theta_E^{T_1}+\cot\theta_E^{T_2}\ge0$, i.e.,
$\theta_E^{T_1}+\theta_E^{T_2}\le\pi$ for the two angles opposite
to each interior edge: this is exactly the defining property of a
\emph{Delaunay triangulation}, which is more general and more
practical than a nonobtuse triangulation. A mesh in which every
simplex is nonobtuse satisfies \eqref{simpicialmesh} in any
dimension.
\end{enumerate}
In both cases the adjacency graph of $\bbS$ contains the
edge--vertex connectivity graph of the mesh, and we assume it is
connected, i.e., $\bbS$ is \emph{irreducible}. This holds
automatically for the finite difference scheme, and for the $P^1$
scheme whenever \eqref{simpicialmesh} holds strictly on enough
edges, e.g., for a strictly Delaunay triangulation, see
\cite{li2019monotonicity}. This is the connectivity required in
Assumption~\ref{asp:struct}, and with it items (i) and (ii) of
Lemma~\ref{lem:disc} hold for both schemes.

Consequently $\bbM A_{\bu}=\bbS+\bbM\bbV+\beta\bbM\diag(\bu^2)$ is an
irreducible symmetric M-matrix for every $\bu$, so $A_{\bu}$ is
monotone and, by the Perron--Frobenius theorem, the smallest
eigenvalue of $A_{\bu}$ is simple and its eigenvector is the only
nonnegative one, up to normalization
\cite[Section~3]{chen2024fully}. This yields the two facts quoted at
the end of Section~\ref{sec:fem}. First, as shown in
\cite[Section~3]{chen2024fully}, the discrete energy inherits the
hidden convexity of the continuous problem, being convex as a
function of the discrete density, so that \eqref{gs-h} has a
minimizer that is unique up to sign, and its positive representative
is the discrete ground state $\bu_{\mathrm{GS}}$. Second, if
$\bu\ge0$ is a critical point of \eqref{gs-h}, then $A_{\bu}\bu
=\lambda_h\bu$ with $\bu\ge0$ and $\bu\neq0$, so by
Perron--Frobenius $\bu$ must be the eigenvector belonging to the
smallest eigenvalue of $A_{\bu}$. In particular $\bu>0$, and the
convexity in the density variables then forces
$\bu=\bu_{\mathrm{GS}}$. The same conclusions were obtained for
mass-lumped $P^1$ finite elements on unstructured meshes in
\cite{hauck2024positivity} by the inverse positivity of irreducible
M-matrices and a discrete Picone inequality, avoiding the explicit
use of the Perron--Frobenius theorem.

\section{The Davis--Yin splitting}
\label{app:dys}

This appendix recalls the Davis--Yin splitting and
derives the two schemes \eqref{dys} and \eqref{dys2} of
Section~\ref{sec:scheme}. Throughout, the
proximal operator of a function $\varphi$ with parameter $\gamma>0$
is taken with respect to the discrete inner product
\eqref{discreteL2norm},
\begin{equation}
\mathrm{prox}_\varphi^\gamma(\bv)
:=\operatorname{argmin}_{\bw\in\bR^N}\ \varphi(\bw)
+\frac1{2\gamma}\norm{\bw-\bv}_2^2 .
\label{prox}
\end{equation}

\subsection{The three-term splitting}
The discrete ground state problem \eqref{gs-h} is the unconstrained
minimization
\begin{equation}
\min_{\bu\in\bR^N}\ F(\bu)+G(\bu)+H(\bu),
\label{threeterm}
\end{equation}
of the three functions
\begin{equation}
F(\bu):=\tfrac12\langle-\Delta_h\bu,\bu\rangle_h,
\qquad
G(\bu):=\iota_{\mathcal M}(\bu),
\qquad
H(\bu):=\tfrac12\langle\bbV\bu,\bu\rangle_h
+\tfrac\beta4\norm{\bu}_4^4 ,
\label{fgh}
\end{equation}
where $F$ is the discrete Dirichlet (kinetic) energy, $H$ collects
the potential and interaction energies, so that $F+H=\bE_h$ is the
discrete energy \eqref{fd-energy}, and $\iota_{\mathcal M}$ is the
indicator function of $\mathcal M$, equal to $0$ on $\mathcal M$ and
$+\infty$ elsewhere, which enforces the normalization constraint.

For $F$, the first-order condition for the minimization in
\eqref{prox} is $\gamma(-\Delta_h\bw)+\bw-\bv=0$, so that
\begin{equation}
\mathrm{prox}_F^\gamma(\bv)=(\bbI-\gamma\Delta_h)^{-1}\bv
=(\bbM+\gamma\bbS)^{-1}\bbM\bv ,
\label{proxf}
\end{equation}
i.e., one solve of a linear system with the symmetric positive
definite matrix $\bbM+\gamma\bbS$. For $G$, the proximal operator of
the indicator function of a set is the projection onto that set, and
the projection onto $\mathcal M$ is the normalization
\begin{equation}
\mathrm{prox}_G^\gamma(\bv)=\frac{\bv}{\norm{\bv}_2},
\qquad \bv\neq0,
\label{proxg}
\end{equation}
which is independent of $\gamma$ and costs no linear solve. Finally
$H$ is smooth, with
\begin{equation}
\nabla H(\bu)=\bbV\bu+\beta\bu^3 ,
\label{gradh}
\end{equation}
evaluated entrywise at no cost.

\subsection{The Davis--Yin iteration}
For a sum of three functions as in \eqref{threeterm}, in which two
terms are handled by their proximal operators and the third by its
gradient, the Davis--Yin splitting \cite{davis2017three} reads
\begin{equation}
\begin{cases}
\bu^k=\mathrm{prox}_G^\gamma(\bz^k),\\[2pt]
\bxit^k=\mathrm{prox}_F^\gamma
\bigl(2\bu^k-\bz^k-\gamma\nabla H(\bu^k)\bigr),\\[2pt]
\bz^{k+1}=\bz^k+\bxit^k-\bu^k .
\end{cases}
\label{DYS}
\end{equation}
It contains two classical schemes as special cases: for $H\equiv0$ it
reduces to the Douglas--Rachford splitting, and for $G\equiv0$ to the
forward--backward splitting, which is why it is also known as the
forward--Douglas--Rachford splitting \cite{raguet2019note}. For
\emph{convex} $F,G,H$ with $\nabla H$ Lipschitz continuous with
constant $L$, the iteration \eqref{DYS} converges for any step size
$\gamma<2/L$ \cite{davis2017three}. Substituting
\eqref{proxf}--\eqref{gradh} into \eqref{DYS} gives the scheme
\eqref{dys} analyzed in this paper.

An alternative splitting moves part of the potential into $F$.
Write $V=V_1+V_2$ with $V_1,V_2\ge0$, and let $\bbV_1$ and
$\bbV_2$ be the corresponding diagonal matrices, so that
$\bbV=\bbV_1+\bbV_2$. In place of \eqref{fgh} take
\begin{equation}
F(\bu):=\tfrac12\langle(-\Delta_h+\bbV_1)\bu,\bu\rangle_h,
\qquad
G(\bu):=\iota_{\mathcal M}(\bu),
\qquad
H(\bu):=\tfrac12\langle\bbV_2\bu,\bu\rangle_h
+\tfrac\beta4\norm{\bu}_4^4 ,
\label{fgh2}
\end{equation}
so that again $F+H=\bE_h$. The first-order condition for the
minimization in \eqref{prox} now reads
$\gamma(-\Delta_h+\bbV_1)\bw+\bw-\bv=0$, so
\begin{equation}
\mathrm{prox}_F^\gamma(\bv)
=(\bbI-\gamma\Delta_h+\gamma\bbV_1)^{-1}\bv
=(\bbM+\gamma\bbS+\gamma\bbM\bbV_1)^{-1}\bbM\bv ,
\label{proxf2}
\end{equation}
while $\mathrm{prox}_G^\gamma$ is unchanged and
$\nabla H(\bu)=\bbV_2\bu+\beta\bu^3$. Substituting these into the
iteration \eqref{DYS}  gives the scheme \eqref{dys2}. The matrix $\bbM+\gamma\bbS+\gamma\bbM\bbV_1$ differs
from $\bbM+\gamma\bbS$ by a nonnegative diagonal matrix, so it is
still an irreducible Stieltjes matrix and its inverse has strictly
positive entries, by the argument of Lemma~\ref{lem:disc}(ii).
 
\section*{Acknowledgments}

This work used DeltaAI at the National Center for Supercomputing
Applications (NCSA) through allocation MTH260013 from the Advanced
Cyberinfrastructure Coordination Ecosystem: Services \& Support
(ACCESS) program, which is supported by U.S.\ National Science
Foundation grants \#2138259, \#2138286, \#2138307, \#2137603,
and \#2138296.

\section*{Declaration of AI-assisted technologies}

The authors have used generative AI tools such as ChatGPT and Claude
to assist mathematical discussions, algorithm implementations and drafting the manuscript. After
using AI tools, the authors have carefully reviewed, edited and
polished the manuscript, and take full responsibility for the
content.

\bibliographystyle{amsplainrepeat}
\bibliography{references}

\end{document}